\documentclass[a4paper,11pt, reqno]{amsart}

\usepackage[longnamesfirst,semicolon]{natbib}
\usepackage{geometry}
\usepackage{float}
\usepackage{hyperref}
\usepackage[justification=centering]{caption}
\usepackage{subcaption}
\usepackage{amsthm}
\usepackage{amsmath,latexsym}
\usepackage{amssymb}

\usepackage[ruled]{algorithm2e}

\newcommand{\dsum}{\displaystyle\sum}
\newcommand{\dmin}{\displaystyle\min}
\newcommand{\dmax}{\displaystyle\max}

\def\R{\mathbb{R}}
\def\k{\kappa}
\def\sign{\mathrm{sign}}
\def\Z{\mathbb{Z}}
\def\bbeta{\boldsymbol\beta}
\def\d{\mathrm{D}}

\def\omf{{\rm OM}_\lambda}

\usepackage{multirow}

\usepackage{hyperref}

\newtheorem{theo}{Theorem}[section]
\newtheorem{ex}[theo]{Example}
\newtheorem{remark}[theo]{Remark}
\newtheorem{lemma}[theo]{Lemma}

\usepackage{pgfplots}
\pgfplotsset{compat=newest}
\usetikzlibrary{decorations.markings}

\usepgfplotslibrary{patchplots}

\let\origmaketitle\maketitle
\def\maketitle{
  \begingroup
  \def\uppercasenonmath##1{} 
  \let\MakeUppercase\relax 
  \origmaketitle
  \endgroup
}

\begin{document}

\title[Multisource Hyperplanes Location]{\large On the multisource hyperplanes location problem\\to fitting set of points}

\author[V. Blanco, A. Jap\'on, D. Ponce \MakeLowercase{and} J. Puerto]{{\large V\'ictor Blanco$^\dagger$, Alberto Jap\'on$^\ddagger$, Diego Ponce$^\ddagger$  and Justo Puerto$^\ddagger$}\medskip\\
$^\dagger$IEMath-GR, Universidad de Granada\\
$^\ddagger$IMUS, Universidad de Sevilla}

\address{IEMath-GR, Universidad de Granada}
\email{vblanco@ugr.es}

\address{IMUS, Universidad de Sevilla}
\email{ajapon1@us.es}
\email{dponce@us.es}
\email{puerto@us.es}

\begin{abstract}
In this paper we study the problem of locating a given number of hyperplanes minimizing an objective function of the closest distances from a set of points. We propose a general framework for the problem in which norm-based distances between  points and  hyperplanes are aggregated by means of ordered median functions.  A compact Mixed Integer Linear (or Non Linear) programming formulation is presented for the problem and also an extended set partitioning formulation with an exponential number of variables is derived.  We develop a column generation procedure embedded within a branch-and-price algorithm for solving the problem by adequately performing its preprocessing,  pricing and  branching. We also analyze geometrically the optimal solutions of the problem, deriving properties which are exploited to generate initial solutions for the proposed algorithms. Finally, the results of an extensive computational experience are reported. The issue of scalability is also addressed showing theoretical upper bounds on the errors assumed by replacing the original datasets by aggregated versions.
\end{abstract}


\maketitle

\section{Introduction}

Location Analysis deals with the determination of the \textit{optimal} positions of facilities to satisfy the demand of a set of customers. The problems analyzed in the field are diverse  but can be usually classified as:  Discrete Location problems (DLP)  and Continuous Location problems (CLP). In the first family, a set of potential facilities is previously given and the goal is to select, among them, the optimal ones under one or more criteria. The main tools for solving these problems come from Discrete Optimization, or more precisely, from Integer Linear Programming. In the second family of problems, the facilities have to be located in a continuous space and then, convex analysis and global optimization tools are needed to solve the problems. The most popular problem in the latter family  is the Weber problem~\citep{Weber} in which a single point-facility has to be positioned on the plane  so as to minimize the overall sum of the (Euclidean) distances to a set of (planar) demand points. The applications of both types of location problems are vast. DLP are more common in the location of \textit{physical} facilities (as ATMs, supermarkets, stations, etc), while CLP are more useful when locating facilities in telecommunication networks (as wifi routers, servers, etc) or even to provide the sets of potential facilities for a DLP.

In this paper we study a problem that falls into the family of CLP. More specifically,  we focus on the determination of \textit{optimal} hyperplanes fitting a given finite set of demand points.  The location of a single hyperplane is a classical problem that has been addressed in different fields.  On the one hand, this problem clearly extends the classical Weber problem, but where instead of locating zero-dimensional facilities one looks for locating higher dimensional structures. On the other hand, in Statistics and Data Analysis, the determination of a hyperplane  minimizing the sum of squares of vertical residuals is key for estimating a multivariate  linear regression model using the Least Sum of Squares (LSS) method, credited to \cite{Gauss}.  One can also find  recent useful applications, both in Location Science and Data Analysis, for the problem of finding \textit{optimal} hyperplanes fitting a set of points.  For instance, \cite{ER11} deals with the location of a rapid transit line on the plane to be used as an alternative transportation mean. Analogously, the widely used Support Vector Machine (SVM) methodology due to \cite{Vapnik}, is also based on constructing a hyperplane minimizing certain loss functions of the distances  to a given set of points.

Scanning the literature one can find that most of the attention has been devoted to finding hyperplanes with any of the following assumptions (see e.g.,
\citep{MS98,Schobel99,Schobel03,Schobel15,MS01,PC01,BJS02,BJS03,BPS18}):
(a) the problem is embedded on the plane; (b) a single hyperplane has to be located; (c) the vertical distance between each point and the hyperplane is considered;  or (d) the residuals are aggregated by the sum or the maximum operators. Our goal here is to study a  generalization of this problem in which, we construct simultaneously a given number, $p$, of hyperplanes in any finite dimensional space, $\R^d$, by minimizing a rather general globalizing function, an ordered median function, of the residuals from the points to the fitting bodies.  Ordered median functions aggregate the set of distances from the demand points to their closest hyperplanes (residuals) by means of a sorting, weighting averaging operation:  distances  are sorted and then their weighted sum is performed. The sum and  maximum functions can be easily represented as ordered median functions with adequate choices of the weights inducing the median and center objective functions. Also, the $k$-centrum (sum of the $k$-th largest distances) or the cent-dian (convex combination of the sum and the max criteria can be cast within this family of functions. In addition, different point-to-hyperplane norm-based distances are considered as a measure of the residuals of the fitting. Thus, this paper naturally extends the analysis performed in ~\cite{BPS18} where the location of a single ordered median hyperplane was studied.

As in the classical Weber problem~\citep{Weber}, the extension from the location of one to several facilities (the so-called multisource problem) is not trivial~\citep{BPE16}. Actually, while the classical single-facility point location problem  with standard distances ($\ell_\tau$, polyhedral, etc) can be formulated as a  Second Order Cone programming problem~\citep{BPE14} (being then polynomially solvable), its multisource version becomes a non-convex NP-hard problem~\citep{BPE16}.

In the case of locating hyperplanes, the situation is even harder, since the location of a single hyperplane is, in general, an NP-hard problem (see \cite{BPS18}) whose exact solution can only be obtained for vertical and polyhedral norm based residuals.

The problem considered in this paper is not fully new although, in our opinion, it has not been fully analyzed and there is room for further improvement. In particular, similar problems have been analyzed from the Data Analysis field, and different names have been adopted. In the so-called \textit{Clusterwise Linear Regression} (CLR) problem, a set of observations is provided and the goal is to cluster them by means of the sum of the squared residuals of several multivariate regression models 
\citep{spath81,H99,CCH14,PJKW17,CLR18}. In \citep{BS07}, classification and regression are simultaneously performed, and also clustering by classical linear regression approaches. Finally, in \citep{BM00}, the clusters are constructed based on the closest distances to \textit{optimal} hyperplanes in a given $d$-dimensional space. In the so-called \textit{Piecewise Linear Regression} problem, a dependent variable is partitioned into $p$ intervals and it adjusts linear bodies to each of them (see \citep{McGee-Carleton70}). However, only local search heuristic algorithms have been proposed for these problems, alternating clustering and regression techniques sequentially. \cite{CCH14} present a column generation algorithm for the (planar) clusterwise regression problem  with  sum of squared residuals which combined with some heuristic strategies outperforms previous results in the literature. Moreover,  \cite{PJKW17} generalized the clusterwise regression problem by allowing each entity to have more than one observation and propose an exact mathematical programming-based approach relying on column generation, and several heuristics.

The main contributions of this paper are:
\begin{enumerate}
\item To provide a general framework for the simultaneous location of several hyperplanes to fit a data set using mathematical programming tools. We formulate the problem by using general norm-based error measures of the distance from points to hyperplanes and ordered median functions to aggregate the residuals. This approach generalizes both the standard multisource regression~\citep{CCH14,PJKW17} and also the more recent proposal for the $p=1$ case \citep{BPS18}.
\item To develop two exact solution methods:
\begin{enumerate}
\item One based on a compact formulation, that for vertical residuals (resulting in a Mixed Integer Second Order Cone Optimization  problem) and for polyhedral norm-based residuals  (resulting in a Mixed Integer linear Programming Problem) can be solved by using some of the available off-the-shell solvers.
\item A novel branch-and-price algorithm, based on a set partitioning formulation for the problem, combining several features as preprocessing, exact and heuristic pricing, and Ryan-and-Foster branching.
\end{enumerate}
\item To  prove some geometrical characterizations of ordered median optimal hyperplanes that are incorporated in the preprocessing phase of our column generation approach.
\item To compare the proposed approaches on a extensive battery of computational experiments on both real and synthetic instances.
\item To  derive upper bounds on the error assumed by aggregation procedures on original datasets that allow to control the scalability of the proposed approaches.
\end{enumerate}

The rest of the paper is organized as follows. In Section \ref{sec:2} we introduce the problem and fix the notation for the rest of the sections. This section also contains two illustrative examples taken from the literature. Section \ref{sec:3} is devoted to a first compact formulation for the problem. This formulation has a polynomial number of variables and constraints but its performance is not always good since it has a large integrality gap. For that reason, in Section \ref{sec:4} we develop an alternative formulation with an exponential number of variables that is solved (exactly, for vertical and polyhedral-norm based residuals) within a branch-and-price (B\&P) algorithm using column generation at each node of the branching tree. This section describes all the elements of this B\&P: initialization, pricing (exact and heuristic) and branching. Section \ref{sec:comp_result} reports our computational results based on two different datasets: the classical 50 points dataset by \cite{EWC74} and another synthetic dataset randomly generated. Section \ref{sec:6} is devoted to explore scalability issues and finally Section \ref{sec:7} draws some conclusions and future extensions.

\section{Multisource Location of Hyperplanes}
\label{sec:2}

In this section we describe the problem under study and fix the notation for the rest of the paper.

We are given a set of $n$ observations/demand points (denoted as points from now on) in $\R^d$, $\{x_1, \ldots, x_n\} \subset \R^{d}$ and $p\in \Z_+$ ($p>0$). Our goal is to find $p$ hyperplanes in $\R^d$ that minimize an objective function of the closest distances from points to hyperplanes.  We denote the index sets of demand points and hyperplanes by $I=\{1, \ldots, n\}$ and $J=\{1, \ldots, p\}$, respectively. Given $\bbeta\in\R^d$ and $\alpha\in \R$, we denote by
$$
\mathcal{H}(\bbeta,\alpha) = \{y \in \R^d: \bbeta^t y + \alpha=0\},
$$
the hyperplane in $\R^d$ with coefficients $\bbeta$ and intercept $\alpha$ (here $v^t$ stands for the transpose of the vector $v\in \R^d$).

Several elements are involved when finding the \textit{best} $p$ hyperplanes to fit a set of demand points. In what follows we describe them:
\begin{itemize}
\item \textit{Residuals:} The point-to-hyperplane measure of closeness. Given a demand point $x =$ $(x_1$, $\ldots$, $x_d)\in \R^d$ and a hyperplane $\mathcal{H}(\bbeta,\alpha)$, how far/close is the point from the hyperplane? The classical fitting methods use the so-called \textit{vertical-distance} measure, which given a reference coordinate, say   the $d$-th, computes the deviation $x_d + \frac{\alpha}{\beta_d}  + \sum_{\ell=1}^{d-1} \frac{\beta_\ell}{\beta_d} x_{\ell}$ for all $i\in I$, whenever $\beta_d\neq 0$. However, it has been already proposed that the use of more general distance measures based on norms may be advisable. In particular, some authors (see e.g., \cite{BPS18,BPR19}) have shown the usefulness of norm-based distances, such as polyhedral, or $\ell_\tau$-distances ($\tau\geq1$). Among them, we mention, for their importance, the Manhattan ($\ell_1$-norm), the Tchebyshev ($\ell_\infty$-norm) or the Euclidean ($\ell_2$-norm) distances.

Thus, for a point $x\in \R^d$ and a hyperplane $\mathcal{H}(\bbeta,\alpha)$, we consider the residual from  $x$ to $\mathcal{H}(\bbeta,\alpha)$ as:
$$
\varepsilon_x(\bbeta,\alpha) = \d\Big(x, \mathcal{H}(\bbeta,\alpha)\Big) := \min \{\d(x,y): y \in \mathcal{H}(\bbeta,\alpha)\},
$$
where $\d$ is a norm-based distance or the vertical distance in $\R^d$ (see \cite{mangasarian,BPS18} for further details on this projection).

\item \textit{Allocation Rule}:  Given a set of hyperplanes and a point,  once the residuals to each of the hyperplanes are calculated,  one has to allocate the point to a single hyperplane. Different alternatives can be considered, as the allocation to the closest or the  furthest hyperplane. In our framework we assume, as usual in Location Analysis, that each point is allocated to the hyperplane with the smallest residual, i.e., for a point $x\in \R^d$ and an arrangement of hyperplanes  $\mathbb{H} = \Big\{\mathcal{H}(\bbeta_j,\alpha_j): j \in J\Big\}$, the final residual point-to-hyperplanes is computed as:
$$
\varepsilon_x\Big(\mathbb{H}\Big) = \min_{j\in J} \varepsilon_x(\bbeta_j,\alpha_j),
$$
and the hyperplane, $\mathcal{H}(\bbeta_j,\alpha_j)$, reaching such a minimum is the one where $x$ is allocated (in case of ties among more than one hyperplane, a random assignment is performed).

\item \textit{Aggregation of Residuals}:  Given a set of points and an arrangement of hyperplanes, once the residuals are computed with respect to the arrangement, and  in order to find the $p$ hyperplanes that best fit the $n$ data points, a global measure of goodness must be chosen for aggregating the residuals.  The classical aggregation functions are the sum or maximum of squared residuals. Most of these criteria can be cast within the framework of the family of ordered median aggregation criteria. More explicitly, given $x_1, \ldots, x_n\in \R^d$, an arrangement of hyperplanes $\mathbb{H} = \Big\{\mathcal{H}(\bbeta_j,\alpha_j): j \in J\Big\}$, and $\lambda \in \R_+^n$ (with $\lambda_1\geq \cdots \geq \lambda_n\geq 0$) the $\lambda$-ordered median function is defined as:
\begin{equation}\label{omf}\tag{OMF}
   \textsf{OM}_\lambda(\varepsilon_1,\ldots,\varepsilon_n)=  \dsum_{i=1}^n \lambda_i \;e_{(i)},
\end{equation}
where $e_{(1)}, \ldots, e_{(n)}$ are defined such that $e_{(i)} \in \{\varepsilon_{x_1}(\mathbb{H}), \ldots, \varepsilon_{x_n}(\mathbb{H})\}$ for all $i\in I$ and $e_{(1)}\geq \cdots \geq e_{(n)}$. Observe that particular cases of ordered median problem are the sum ($\lambda_i=1$, $i=1, \ldots, n$), the maximum ($\lambda_1=1$, $\lambda_i=0$, $i\neq 1$), the $k$-centrum ($\lambda_i=1$, $i=1,\ldots, k$, $\lambda_j=0$, $j>k$) or the $\rho$-centdian, a convex combination of sum and max criterion ($\lambda_1=1, \lambda_i=\rho, i=2, \ldots, n)$, for $0<\rho<1$.
\end{itemize}

Summarizing all the above considerations,  the Multisource Ordered Median Fitting Hyperplanes Problem (MOMFHP) can be stated as the problem of finding $\bbeta_1, \ldots, \bbeta_p \in \R^{d}$ and $\alpha_1, \ldots, \alpha_p\in \R$ solving the following optimization problem:
\begin{align}
\min & \dsum_{i\in I} \lambda_i \;e_{(i)} \tag{${\rm MOMFHP}_0$}\label{mofhp0}\\
\mbox{s.t. } & e_i \geq \min_{j\in J} \varepsilon_{x_i}(\bbeta_j,\alpha_j), \forall i \in I,\nonumber\\
& \bbeta_j \in \R^d, \alpha_j \in \R, \forall j\in J,\nonumber\\
& e_i \geq 0, \forall i\in I.\nonumber
\end{align}
where $e_i$ represents the residual for the $i$-th point in the data set, for all $i\in I$.

(MOMFHP) appears when different trends or clouds have to be differentiated on the demand points, and then, different hyperplanes want to be use to fitting to the points, such that the global error assumed, when the points are allocated to their closest hyperplanes, is as small as possible. In Figure \ref{fig:0} we illustrate a set of demand points in the plane which could be clustered into three groups according to different linear trends which are drawn in gray color.
\begin{figure}[h]
\begin{center}
\fbox{
\begin{tikzpicture}[scale=1.25]

\definecolor{color0}{rgb}{0.12156862745098,0.466666666666667,0.705882352941177}
\definecolor{color1}{rgb}{1,0.498039215686275,0.0549019607843137}
\definecolor{color2}{rgb}{0.172549019607843,0.627450980392157,0.172549019607843}

\begin{axis}[
hide x axis,
hide y axis,
]
\addplot [only marks, draw=black, fill=black, mark size=1, colormap/viridis]
table{%
x                      y
8.14 9.26
4.87 13.6
1.2 26.79
1.8 24.19
7.48 8.82
3.77 19.14
3.71 15.05
3.49 17.5
1.79 24.88
8.74 2.81
2.61 20.86
6.44 6.71
4.94 14.7
3.46 20.93
6.37 10.52
6.01 12.75
9.26 2.17
9.44 1.18
2.57 23.48
4.79 11.85
9.68 1.14
6.84 8.41
0.79 27.18
5.15 11.71
0.96 26.03
1.86 19.17
7.73 7.75
8.11 4.01
6.14 10.66
7.05 9.36
2.01 24.44
2.2 21.26
9.76 2.68
1.28 24.15
9.8 1.87
0.74 25.04
5.14 12.55
0.64 21.36
5.06 11.84
7.48 5.39
0.44 25.28
2.94 17.68
1.79 22.24
8.19 6.31
9.39 3.87
4.18 16.96
7.46 5.28
6.18 11.35
4.36 16.37
5.14 11.59
6.79 -39.23
3.79 -19.46
5.58 -31.49
9.33 -61.91
1.3 2.09
7.96 -48.14
2.28 -10.98
4.31 -27.02
2.35 -9.05
0.64 3.19
0.35 5.14
1.54 -1.85
6.81 -41.02
2.61 -13.25
5.81 -37.34
0.08 8.07
7.56 -42.95
5.13 -27.31
3.01 -11.78
9.1 -58.83
3.87 -19.65
4.58 -24.65
4.31 -25.93
5.43 -28.07
0.8 4.4
3.79 -21.27
4.98 -28.81
2.09 -5.63
5.2 -25.96
7.6 -47.53
4.58 -29.88
5.05 -26.66
5.01 -28.81
8.25 -49.8
1.87 -2.78
4.93 -28.52
2.3 -8.36
4.64 -22.92
6.24 -38.15
6.76 -39.43
7.47 -45.49
1.16 -1.55
1.55 -1.19
4.72 -24
4.97 -29.9
9.23 -62.74
1.51 0.64
2.49 -4.72
6.86 -40.08
4.12 -21.63
6.85 11.03
9.21 17.05
3.78 10.04
2.03 -0.59
5.94 9.26
5.39 12.62
3.03 5.27
0.42 -0.92
5.15 7.23
1.54 -0.7
6.35 7.01
7.37 14.16
8.2 14.17
2.36 3.5
3.7 8.7
5.19 6.79
7.84 14.03
2.31 1.75
4.91 6.27
7.94 14.93
2.94 3.26
4.99 9.36
0.38 -0.57
9.23 15.31
1.47 1.91
4.17 5.85
5.35 9.2
1.57 2.05
5.3 8.5
8.85 19.75
3.95 6.13
6.67 10.26
6.99 12.75
9.41 19.3
7.09 12.54
7.33 11.95
3 5.13
7.9 14.68
1.41 4.64
3.2 5.47
3.07 1.32
3.65 5.64
6.81 14.76
7 12.21
0.61 -0.29
4.97 7.36
9.3 15.05
0.76 1.14
2.57 3.01
1.04 -1.44
};
\addplot [semithick, gray]
table {%
0 27.79
10 -0.01
};
\addplot [semithick, gray]
table {%
0 10.37
10 -65.73
};
\addplot [semithick, gray]
table {%
0 -2.37
10 19.13
};
\end{axis}

\end{tikzpicture}}
\end{center}
\caption{Illustration of a feasible solution of our problem for a set of demand points.\label{fig:0}}
\end{figure}
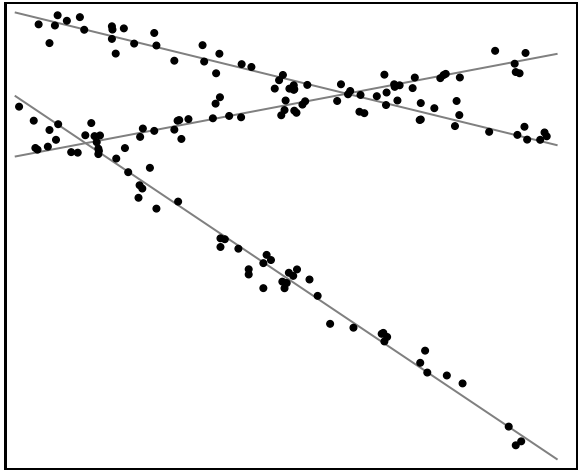
In the following example we illustrate the problem under analysis in two classical instances.

\begin{ex}
In the seminal paper by \cite{McGee-Carleton70}, the authors illustrate the Clusterwise Linear Regression method with two instances. The first instance, \citep{Quandt}, consists of $20$ points on the plane, $\{x_1, \ldots, x_{20}\}$ generated as follows:
\begin{align*}
x_{i2} = 2.5 + 0.7 x_{i1} + \epsilon_i, \mbox{\rm for $i=1, \ldots, 12$, and}\\
x_{i2} = 5 + 0.5 x_{i1} + \epsilon_i, \mbox{\rm for $i=13, \ldots, 20$,}
\end{align*}
where $\epsilon$ is randomly generated as a univariate normal distribution with mean $0$ and standard deviation $1$.

We run our model with this dataset choosing as residuals the $\ell_1$-norm projection of the data onto the hyperplanes, and four different ordered median criteria:  Weber, Center, $\lceil \frac{n}{2}\rceil$-Centrum ($\lambda=(\overbrace{1, \ldots, 1}^{\lceil \frac{n}{2}\rceil}, 0, \ldots, 0)$) and $0.9$-centdian ($\lambda=(1, 0.9, \ldots, 0.9)$). The results are shown in Figure \ref{fig:q1}.
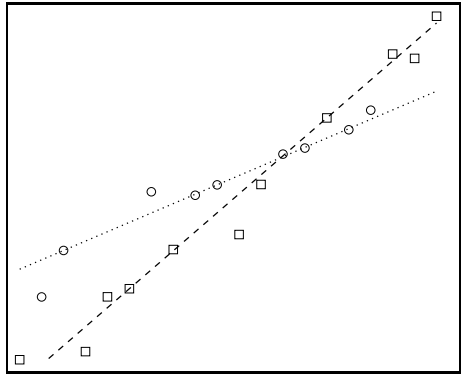
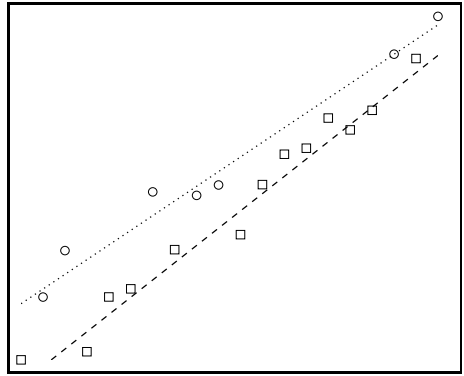
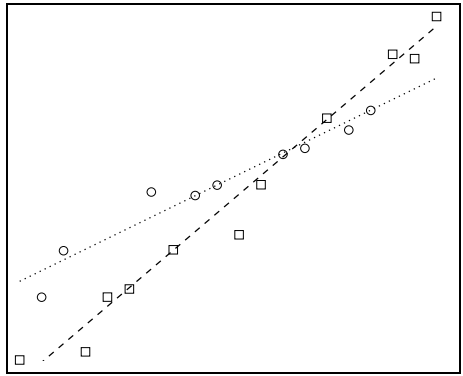
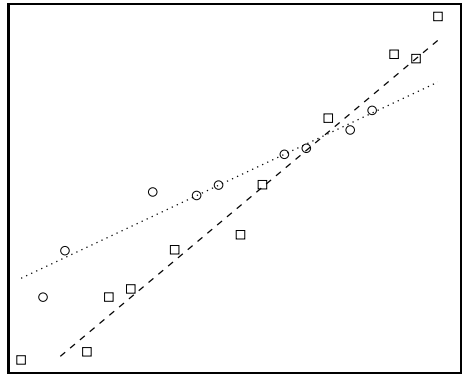
\begin{figure}[ht!]
	\centering
	\begin{subfigure}{0.45\textwidth} 
		\fbox{
\begin{tikzpicture}[scale=0.8]

\begin{axis}[
hide x axis,
hide y axis,
tick align=outside,
tick pos=left,
x grid style={white!69.01960784313725!black},
xmin=1, xmax=20,
xtick style={color=black},
y grid style={white!69.01960784313725!black},
ymin=3.14, ymax=17.197,
ytick style={color=black}
]
\addplot [only marks, draw=black, fill=black, colormap/viridis,mark=square]
table{%
x                      y
4 3.473
5 5.714
6 6.046
8 7.65
1 3.14
12 10.312
20 17.197
15 13.036
11 8.264
19 15.472
18 15.65
};
\addplot [only marks, draw=black, fill=black, colormap/viridis, mark=o]
table{%
x                      y
13 11.555
2 5.71
17 13.353
3 7.612
14 11.802
16 12.551
10 10.296
7 10.014
9 9.871
};
\addplot [semithick, black, dashed]
table {%
1 2.16266666666667
20 16.9193333333333
};
\addplot [semithick, black, dotted]
table {%
1 6.84514285714286
20 14.1302857142857
};
\end{axis}

\end{tikzpicture}}
		\caption{Solution for Weber criterion.} 
	\end{subfigure}
	\vspace{0.2em} 
	\begin{subfigure}{0.45\textwidth} 
		\fbox{
\begin{tikzpicture}[scale=0.8]

\begin{axis}[
hide x axis,
hide y axis,
tick align=outside,
tick pos=left,
x grid style={white!69.01960784313725!black},
xmin=1, xmax=20,
xtick style={color=black},
y grid style={white!69.01960784313725!black},
ymin=3.14, ymax=17.197,
ytick style={color=black}
]
\addplot [only marks, draw=black, fill=black, colormap/viridis,mark=square]
table{%
x                      y
4 3.473
13 11.555
5 5.714
6 6.046
8 7.65
1 3.14
12 10.312
17 13.353
15 13.036
11 8.264
14 11.802
16 12.551
19 15.472
};
\addplot [only marks, draw=black, fill=black, colormap/viridis, mark=o]
table{%
x                      y
2 5.71
20 17.197
3 7.612
10 10.296
7 10.014
18 15.65
9 9.871
};
\addplot [semithick, black, dashed]
table {%
1 2.16771428571429
20 15.598
};
\addplot [semithick, black, dotted]
table {%
1 5.43871428571429
20 16.8482142857143
};
\end{axis}

\end{tikzpicture}}
		\caption{Solution for Center criterion.} 
	\end{subfigure}
		\begin{subfigure}{0.45\textwidth} 
		\fbox{
\begin{tikzpicture}[scale=0.8]

\begin{axis}[
hide x axis,
hide y axis,
tick align=outside,
tick pos=left,
x grid style={white!69.01960784313725!black},
xmin=1, xmax=20,
xtick style={color=black},
y grid style={white!69.01960784313725!black},
ymin=3.14, ymax=17.197,
ytick style={color=black}
]
\addplot [only marks, draw=black, fill=black, colormap/viridis,mark=square]
table{%
x                      y
4 3.473
5 5.714
6 6.046
8 7.65
1 3.14
12 10.312
20 17.197
15 13.036
11 8.264
19 15.472
18 15.65
};
\addplot [only marks, draw=black, fill=black, colormap/viridis, mark=o]
table{%
x                      y
13 11.555
2 5.71
17 13.353
3 7.612
14 11.802
16 12.551
10 10.296
7 10.014
9 9.871
};
\addplot [semithick, black, dashed]
table {%
1 2.28069230769231
20 16.8025384615385
};
\addplot [semithick, black, dotted]
table {%
1 6.36036094674556
20 14.678201183432
};
\end{axis}

\end{tikzpicture}}
		\caption{Solution for $k$-Centrum criterion.} 
	\end{subfigure}
	\vspace{0.2em} 
	\begin{subfigure}{0.45\textwidth} 
		\fbox{
\begin{tikzpicture}[scale=0.8]

\begin{axis}[
hide x axis,
hide y axis,
tick align=outside,
tick pos=left,
x grid style={white!69.01960784313725!black},
xmin=1, xmax=20,
xtick style={color=black},
y grid style={white!69.01960784313725!black},
ymin=3.14, ymax=17.197,
ytick style={color=black}
]
\addplot [only marks, draw=black, fill=black, colormap/viridis,mark=square]
table{%
x                      y
4 3.473
5 5.714
6 6.046
8 7.65
1 3.14
12 10.312
20 17.197
15 13.036
11 8.264
19 15.472
18 15.65
};
\addplot [only marks, draw=black, fill=black, colormap/viridis, mark=o]
table{%
x                      y
13 11.555
2 5.71
17 13.353
3 7.612
14 11.802
16 12.551
10 10.296
7 10.014
9 9.871
};
\addplot [semithick, black, dashed]
table {%
1 1.94430769230769
20 16.2235384615385
};
\addplot [semithick, black, dotted]
table {%
1 6.48302797202797
20 14.5136503496503
};
\end{axis}

\end{tikzpicture}}
		\caption{Solution for Centdian criterion.} 
	\end{subfigure}
	\caption{Lines obtained for Quandt dataset for $\ell_1$-norm residuals and different criteria.\label{fig:q1}} 
\end{figure}

\cite{McGee-Carleton70} also analyzed a real instance, the \texttt{Boston} dataset. It was motivated by the fact that regional stock exchanges were hurt by the abolition of give-ups  in 1968. The model tries to analyze the dollar volume of sales on the Boston Stock Exchange with respect to dollar volumes for the New York and American Stock Exchanges, based on a dataset with $35$ monthly observations from January 1967 to November 1969.  One can observe, in the results  shown in  (Figure \ref{fig:b1}), that our models are able to adequately cast the trends of these observations.
    \begin{figure}[ht!]
	\centering
	\begin{subfigure}{0.45\textwidth} 
		\fbox{
\begin{tikzpicture}[scale=0.8]

\begin{axis}[
hide x axis,
hide y axis,
tick align=outside,
tick pos=left,
x grid style={white!69.01960784313725!black},
xmin=10234.3, xmax=18464.3,
xtick style={color=black},
y grid style={white!69.01960784313725!black},
ymin=57.4, ymax=271.8,
ytick style={color=black}
]
\addplot [only marks, draw=black, fill=black, colormap/viridis,mark=square]
table{%
x                      y
10581.6 78.8
10234.3 69.1
13299.5 87.6
10746.5 72.8
13310.7 79.4
12835.5 85.6
12194.2 75
12860.4 85.3
11955.6 86.9
13351.5 107.8
13285.9 128.7
13784.4 134.5
16336.7 148.7
11040.5 94.2
16056.4 154.1
18464.3 191.3
17092.2 191.9
15178.8 159.6
17436.9 139.4
16482.2 106
13905.4 112.1
11973.7 103.5
12573.6 92.5
16566.8 116.9
13558.7 78.9
11530.9 57.4
11278 75.9
11263.7 109.8
15649.5 129.2
12197.1 115.1
};
\addplot [only marks, draw=black, fill=black, colormap/viridis, mark=o]
table{%
x                      y
11525.3 128.1
12774.8 185.5
12377.8 178
16856.3 271.8
14635.3 212.3
};
\addplot [semithick, black, dashed]
table {%
10234.3 66.4276995156542
18464.3 168.817455443351
};
\addplot [semithick, black, dotted]
table {%
10234.3 131.783192453755
18464.3 305.799852995222
};
\end{axis}

\end{tikzpicture}}
		\caption{Solution for Weber criterion.} 
	\end{subfigure}
	\vspace{0.5em} 
	\begin{subfigure}{0.45\textwidth} 
		\fbox{
\begin{tikzpicture}[scale=0.8]

\begin{axis}[
hide x axis,
hide y axis,
tick align=outside,
tick pos=left,
x grid style={white!69.01960784313725!black},
xmin=10234.3, xmax=18464.3,
xtick style={color=black},
y grid style={white!69.01960784313725!black},
ymin=57.4, ymax=271.8,
ytick style={color=black}
]
\addplot [only marks, draw=black, fill=black, colormap/viridis, mark=square]
table{%
x                      y
10581.6 78.8
10234.3 69.1
13299.5 87.6
10746.5 72.8
13310.7 79.4
12835.5 85.6
12194.2 75
12860.4 85.3
11955.6 86.9
13351.5 107.8
13285.9 128.7
13784.4 134.5
16336.7 148.7
16056.4 154.1
18464.3 191.3
17092.2 191.9
15178.8 159.6
17436.9 139.4
16482.2 106
13905.4 112.1
11973.7 103.5
12573.6 92.5
16566.8 116.9
13558.7 78.9
11530.9 57.4
11278 75.9
15649.5 129.2
};
\addplot [only marks, draw=black, fill=black, colormap/viridis, mark=o]
table{%
x                      y
11525.3 128.1
12774.8 185.5
12377.8 178
16856.3 271.8
14635.3 212.3
11040.5 94.2
11263.7 109.8
12197.1 115.1
};
\addplot [semithick, black, dashed]
table {%
10234.3 56.646132795745
18464.3 172.573558376941
};
\addplot [semithick, black, dotted]
table {%
10234.3 94.4516539537136
18464.3 266.824993241422
};
\end{axis}

\end{tikzpicture} }
		\caption{Solution for Center criterion.} 
	\end{subfigure}
		\begin{subfigure}{0.45\textwidth} 
		\fbox{
\begin{tikzpicture}[scale=0.8]

\begin{axis}[
hide x axis,
hide y axis,
tick align=outside,
tick pos=left,
x grid style={white!69.01960784313725!black},
xmin=10234.3, xmax=18464.3,
xtick style={color=black},
y grid style={white!69.01960784313725!black},
ymin=57.4, ymax=271.8,
ytick style={color=black}
]
\addplot [only marks, draw=black, fill=black, colormap/viridis,mark=square]
table{%
x                      y
10581.6 78.8
10234.3 69.1
13299.5 87.6
10746.5 72.8
13310.7 79.4
12835.5 85.6
12194.2 75
12860.4 85.3
11955.6 86.9
13351.5 107.8
13285.9 128.7
13784.4 134.5
16336.7 148.7
11040.5 94.2
16056.4 154.1
18464.3 191.3
17092.2 191.9
15178.8 159.6
17436.9 139.4
16482.2 106
13905.4 112.1
11973.7 103.5
12573.6 92.5
16566.8 116.9
13558.7 78.9
11530.9 57.4
11278 75.9
15649.5 129.2
12197.1 115.1
};
\addplot [only marks, draw=black, fill=black, colormap/viridis, mark=o]
table{%
x                      y
11525.3 128.1
12774.8 185.5
12377.8 178
16856.3 271.8
14635.3 212.3
11263.7 109.8
};
\addplot [semithick, black, dashed]
table {%
10234.3 69.2556948902485
18464.3 167.538555393847
};
\addplot [semithick, black, dotted]
table {%
10234.3 108.617038991023
18464.3 330.461158293217
};
\end{axis}

\end{tikzpicture} }
		\caption{Solution for $k$-Centrum criterion.} 
	\end{subfigure}
	\vspace{0.5em} 
	\begin{subfigure}{0.45\textwidth} 
		\fbox{
\begin{tikzpicture}[scale=0.8]

\begin{axis}[
hide x axis,
hide y axis,
tick align=outside,
tick pos=left,
x grid style={white!69.01960784313725!black},
xmin=10234.3, xmax=18464.3,
xtick style={color=black},
y grid style={white!69.01960784313725!black},
ymin=57.4, ymax=271.8,
ytick style={color=black}
]
\addplot [only marks, draw=black, fill=black, colormap/viridis,mark=square]
table{%
x                      y
10581.6 78.8
10234.3 69.1
13299.5 87.6
10746.5 72.8
13310.7 79.4
12835.5 85.6
12194.2 75
12860.4 85.3
11955.6 86.9
13351.5 107.8
13285.9 128.7
13784.4 134.5
16336.7 148.7
11040.5 94.2
16056.4 154.1
18464.3 191.3
17092.2 191.9
15178.8 159.6
17436.9 139.4
16482.2 106
13905.4 112.1
11973.7 103.5
12573.6 92.5
16566.8 116.9
13558.7 78.9
11530.9 57.4
11278 75.9
11263.7 109.8
15649.5 129.2
12197.1 115.1
};
\addplot [only marks, draw=black, fill=black, colormap/viridis, mark=o]
table{%
x                      y
11525.3 128.1
12774.8 185.5
12377.8 178
16856.3 271.8
14635.3 212.3
};
\addplot [semithick, black, dashed]
table {%
10234.3 66.3431274521165
18464.3 170.091782409323
};
\addplot [semithick, black, dotted]
table {%
10234.3 131.783192453755
18464.3 305.799852995222
};
\end{axis}

\end{tikzpicture}}
		\caption{Solution for Centdian criterion.} 
	\end{subfigure}
	\caption{Lines obtained for Boston dataset for $\ell_1$-norm residuals and different criteria.\label{fig:b1}} 
\end{figure}
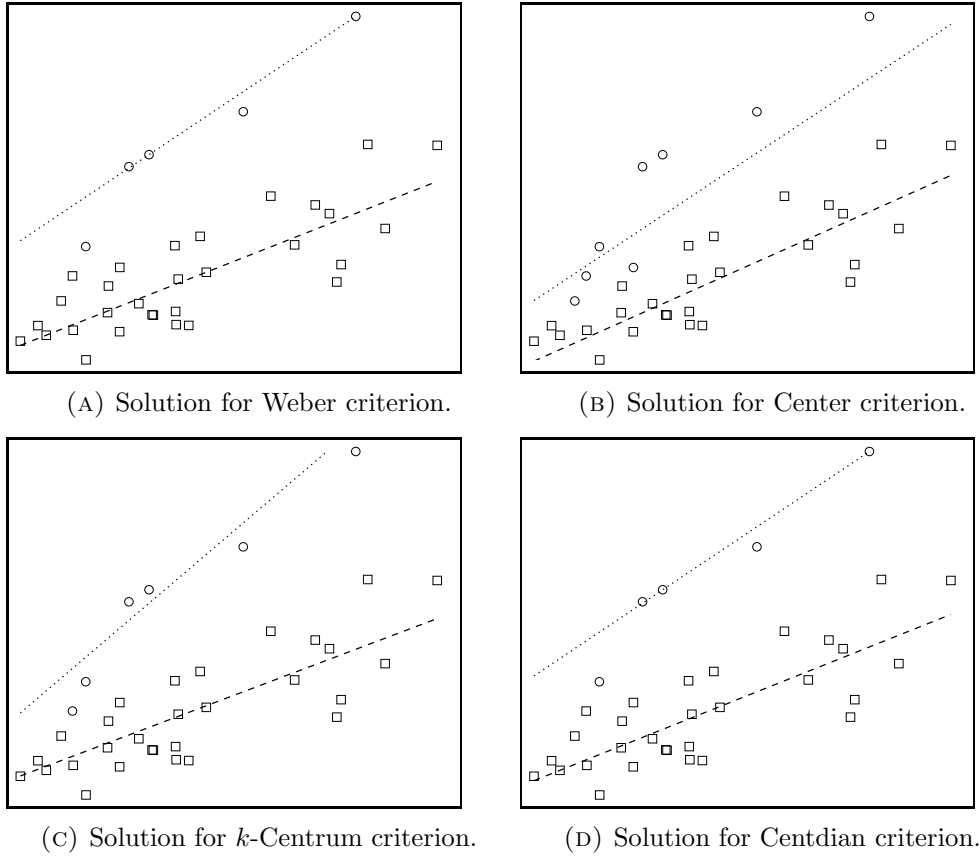

\end{ex}

\section{A Compact Formulation for (MOMFHP$_0$)}\label{sec:3}

In this section we provide a mathematical programming formulation for \eqref{mofhp0}. The main components which involve decisions in this problem, and that have to be adequately included in a suitable formulation, are the representation of general norm-based residuals and the aggregation of residuals using an ordered median function. We describe here how to incorporate all these elements into a mathematical programming formulation which in many cases is suitable to be solved with any of the available MILP/MISOCO solvers.

\begin{theo}
Let $\{x_1, \ldots, x_n\} \subseteq \R^d$, $p \in \Z_+$ ($p>0$) and $ \lambda_1 \ge \cdots \ge \lambda_n\ge 0$. Then, \eqref{mofhp0} can be equivalently reformulated as follows:
\begin{align}
\min &\dsum_{k \in I} u_k +\dsum_{i \in I} v_i \tag{${\rm MOMFHP}$}\label{FHBEP}\\
\mbox{s.t. } & u_k + v_i \geq \lambda_k e_i, \forall i, k\in I,\label{eqq:3}\\
& e_i \geq  \varepsilon_{x_i}(\bbeta_j,\alpha_j)  - M_{ij}(1- z_{ij}), \forall i\in I, j\in J,\label{eqq:1}\\
& \dsum_{j=1}^p z_{ij}=1, \forall i\in I,\label{eqq:4}\\
& z_{ij} \in \{0,1\}, \forall i\in I, j\in J,\label{eqq:5}\\
& e_i \in \R_+, \forall i\in I,\label{eqq:6}\\
& \bbeta_{j} \in \R^d, \alpha_j\in \R, \forall j\in J, \label{eqq:7}\\
& u_k, v_i \in \R, \forall i,k \in I.\label{eqq:8}
\end{align}
where $M_{ij}$ are upper bounds on the residual values $\varepsilon_{x_i}(\bbeta_j,\alpha_j)$, for all $i \in I, j \in J$.
\end{theo}

\begin{proof}
First, observe that given a set of residuals $e_1, \ldots, e_n \geq 0$, the evaluation of the objective function in \eqref{mofhp0} requires sorting  and averaging them (the residuals) with the $\lambda$-weights. In \cite{BPE14}, the authors proved that the computation of $\dsum_{k=1}^n \lambda_k e_{(k)}$  can be done by means of the optimal value of the following Linear Programming Problem (see \cite{BPE14}):
\begin{equation*} \label{eq:BHP} \sum_{k\in I} \lambda_k e_{(k)} =
\left\{
\begin{array}{rrrl}
\min        &\dsum_{k \in I} u_k +\dsum_{i \in I} v_i  \\
 \mbox{s.t.} & u_k+ v_i\ge \lambda_k e_i & \forall k,i \in I,\\
 & u, v \in \R^n.
\end{array}
\right.
\end{equation*}
Thus, the objective function in \eqref{mofhp0} can be replaced by the above objective function and the constraints incorporated to the rest of constraints in the model.

In order to identify the point-to-hyperplane allocation we consider the following set of binary variables:
$$
z_{ij} = \left\{\begin{array}{cl}
1 & \mbox{ if the $i$-th observation is assigned to $\mathcal{H}(\bbeta_j,\alpha_j)$,}\\
0 & \mbox{otherwise,}
\end{array}\right.
$$
for all $i\in I$ and $j\in J$.

Note that with our allocation rule, an observation can be always assigned to a hyperplane that reaches the minimum residual among all the possible assignments  to the $p$ hyperplanes.

Finally, using the variables previously described, the objective function computes the ordered median function of the residuals. Constraints \eqref{eqq:1} assure the correct definition of the residuals $e_i$ and the allocation to their correct hyperplane. Indeed, if $z_{ij}=1$ this constraint forces $e_i$ to take the value of $ \varepsilon_{x_i}(\bbeta_j,\alpha_j)$. Constraints \eqref{eqq:4} assure that only one of these variables will be equal to 1, which in turns forces by the minimization character of the objective function to be the one with  the correct assignment. Finally, \eqref{eqq:5}--\eqref{eqq:8} are the domains of the variables.
\end{proof}

\begin{remark}
Observe that the different choices of ordered median functions are embedded into constraint \eqref{eqq:3}. In some particular cases, this formulation can be simplified avoiding useless variables and constraints.

\begin{itemize}
\item {\bf $p$-Median Problem} ($\lambda=(1, \ldots, 1)$) In this case, since the ordering does not affect the aggregation operator, the $u$ and $v$-variables can be avoided, and the problem simplifies to:
\begin{align*}
\min &\;\;\; \dsum_{i\in I} e_i \nonumber\\
\mathrm{s.t. } &\;\; \eqref{eqq:1}-\eqref{eqq:7}.
\end{align*}
\item {\bf $p$-Center Problem} ($\lambda=(1,0, \ldots, 0)$): For the Center problem, one can represent the objective function, $\max_{i\in I} e_i$, by using an auxiliary variable, $t$, in the usual manner:
\begin{align*}
\min & \;\; t\nonumber\\
\mathrm{s.t. } & \;\; \eqref{eqq:1}-\eqref{eqq:7},\\
 & t \geq e_i, \forall i\in I,
\end{align*}
\item {\bf $p$-$k$-Center Problem} ($\lambda=(\overbrace{1,\ldots,1}^k, 0, \ldots, 0)$): For the $k$-Centrum problem, in \cite{OT03} the authors derive a formulation similar to the one for the center problem:
\begin{align*}
\min & \;\; k\, t +  \dsum_{i\in I} r_i \nonumber\\
\mathrm{s.t. } & \;\; \eqref{eqq:1}-\eqref{eqq:7},\\
 & r_i \geq e_i - t, \forall i\in I, \\
& t \geq 0,\\
& r_i \geq 0, \forall i \in I.
\end{align*}
\end{itemize}
\end{remark}
Note also that the explicit expression of $\varepsilon_{x_i}(\bbeta_j,\alpha_j)$ and then, the difficulty of the optimization problem above, depends (apart from the binary variables that appears in the problem) on the choice of the distance measure $\d$ which defines the residuals of the fitting.  In what follows we describe general shapes for the distances inducing the residuals and how they can be incorporated to \eqref{FHBEP}.

\subsection{Vertical Distance Residuals}

Although not rigorously a distance measure, the so-called \textit{vertical distance} is a very common measure for computing the residuals in Data Analysis. The vertical distance is computed as the absolute deviation, with respect to one of the coordinates, of the hyperplane. Without loss of generality, we consider that the deviation is computed with respect to the $d$-th coordinate, and then, one can assume that $\beta_{jd}=-1$ for $j \in J$. Given $x\in \R^d$ and an hyperplane $\mathcal{H}(\alpha,\bbeta)$ the vertical distance residual is calculated as:
$$
\varepsilon_{x}(\bbeta,\alpha) = \left| x_{d} - \alpha- \dsum_{\ell=1}^{d-1} \bbeta_\ell x_{\ell}\right|.
 $$
This measure can be incorporated to \eqref{FHBEP}, replacing  \eqref{eqq:1} by the following set of linear constraints:
\begin{align*}
e_i \geq x_{id} - \alpha_{j}- \dsum_{\ell=1}^{d-1} \bbeta_{j\ell} x_{i\ell} - M_{ij}(1- z_{ij}), \forall  i\in I, j\in J,\\
e_i \geq -x_{id} + \alpha_{j}+ \dsum_{\ell=1}^{d-1} \bbeta_{j\ell} x_{i\ell} - M_{ij}(1- z_{ij}), \forall  i\in I, j\in J.
\end{align*}
Thus, becoming \eqref{FHBEP} a Mixed Integer Linear Programming problem.

\begin{remark}[Support Vector Regression]

One particular case in which vertical residuals are used in Machine Learning tools is in Support Vector Regression (SVR). \cite{Vapnik13} proposed this methodology for obtaining regression models based on Support Vector Machines as introduced in \cite{Vapnik}. The method is based on fitting a hyperplane to the set of points $\{x_1, \ldots, x_n\}$ with a modified vertical distance, such that only the residuals greater than a given threshold $\Delta\geq 0$ are accounted, apart from maximizing the separation between the observations at each of the sides of the hyperplanes. SVR can be modeled as follows:
\begin{align*}
\min & \dfrac{1}{2} \|\bbeta\|^2_2 + C \dsum_{i \in N} e_i\\
\mbox{\rm s.t. } & e_i \geq \left| x_{id} - \dsum_{k=1}^{d-1} \bbeta_k x_{ik} - \alpha\right| - \Delta, \forall i\in I,\\
& \bbeta \in \R^{d-1}, \alpha \in \R,\\
& e_i \geq 0, \forall i\in I,
\end{align*}
where $C$ is a given parameter.

Observe that the measure used in this approach is nothing but a truncated version of the vertical distance:
$$
 \varepsilon_{x}(\bbeta,\alpha) = \left\{ \begin{array}{cl} | x_{d} - \alpha - \dsum_{\ell=1}^{d-1} \bbeta_\ell x_{\ell}| & \mbox{if  $| x_{d} - \alpha - \dsum_{\ell=1}^{d-1} \bbeta_\ell x_{\ell}| > \Delta$,}\\
 0 & \mbox{otherwise.}
 \end{array}\right.
 $$
Thus, this shape of the residuals can also be embedded in our multisource framework, just by adding to the objective functions the terms measuring the norms of the coefficients of the hyperplanes, i.e., replacing the objective function in \eqref{FHBEP} by
 $$
 \dfrac{1}{2} \dsum_{j\in J} \|\bbeta_j\|^2_2 + \dsum_{i\in I} \lambda_i \;e_{(i)}.
$$
In case  $p=1$ and $\lambda=(1, \ldots, 1)$, we obtain classical SVR taking also into account the parameter $\Delta$, but more flexible counterparts can be generated with our framework.
\end{remark}

\subsection{Norm-based Residuals}

For general norm-based distances,  a given observation $y^t= (y_{1}, \ldots, y_{d})$ and a set of $p$ hyperplanes defined by $\bbeta_1, \ldots, \bbeta_p \in \R^{d}$ and $\alpha_1, \ldots, \alpha_p\in \R$  inducing the arragement $\mathbb{H} = \Big\{\mathcal{H}(\bbeta_j,\alpha_j): j \in J\Big\}$, based on \cite[Theorem 2.1]{mangasarian}, the projection, $\hat y$, of $y$ consistent with the residual $\varepsilon$ induced by a norm $\|\cdot\|$ is
\begin{equation*}
\hat y= y_{-0} -\min_{j\in J} \frac{\alpha_{j}+\bbeta_{j}^t y}{\|(\bbeta_{j1}, \ldots,\bbeta_{jd})\|^*}\k(\bbeta_{j}),
\end{equation*}
where $\|\cdot \|^*$ is the dual norm of $\| \cdot \|$ and $\k(\bbeta)= \displaystyle \mbox{\rm arg }\dmax _{\|z\|=1}  (\bbeta_{j1}, \ldots,\bbeta_{jd})^t z$. Moreover, the residuals can be written as:
\begin{equation} \label{eq:normdist}
\varepsilon_y(\mathbb{H})= \min_{j\in J}  \frac{|\alpha_j+\bbeta_{j}^t y|}{\|(\bbeta_{j1}, \ldots,\bbeta_{jd})\|^*}.
    \end{equation}

\begin{remark}[$\ell_1$-norm case]\label{rem3}
In the  case of the $\ell_1$-norm residuals, the expression above reduces to:
\begin{equation}\label{res:l1}
\varepsilon_{x_i}(\bbeta_j,\alpha_j)= \min_{j\in J}  \frac{|\alpha_j+\bbeta_{j}^t x_i|}{ \dmax_{\ell=1, \ldots, d} |\bbeta_{j\ell}|},
\end{equation}
and constraints \eqref{eqq:1} can be replaced in \eqref{FHBEP} by:
\begin{align}
&e_i \geq \alpha_{j} + \dsum_{\ell=1}^d \bbeta_{j\ell} x_{i\ell} - M_{ij} (1-z_{ij}), \forall i\in I, \forall j\in J\label{eqq:l1-1}\\
& e_i \geq -\alpha_{j} - \dsum_{\ell=1}^d \bbeta_{j\ell} x_{i\ell} - M_{ij} (1-z_{ij}), \forall i\in I, \forall j\in J\label{eqq:l1-2}\\
& \bbeta_{j\ell} = \eta_{j\ell}^+-\eta_{j\ell}^-, \forall j\in J, \ell = 1,\ldots,d,\label{eqq:l1-3}\\
& \eta_{j\ell}^+\leq U_{j\ell}\; \xi_{j\ell}, \forall j\in J,\ell=1,\ldots,d, \label{eqq:l1-4}\\
& \eta_{j\ell}^-\leq U_{j\ell}\;(1-\xi_{j\ell}), \forall j\in J,\ell=1,\ldots,d,\label{eqq:l1-5}\\
& \theta_{j\ell} = \eta^+_{j\ell} + \eta^-_{j\ell}, \forall j\in J,\ell=1,\ldots,d,\label{eqq:l1-6}\\
&  \theta_{j\ell} \leq  1, \forall j\in J,\ell=1,\ldots,d,\label{eqq:l1-7}\\
& \theta_{j\ell} \geq  \mu_{j\ell}, \forall j\in J,\ell=1,\ldots,d,\label{eqq:l1-8}\\
& \dsum_{\ell=1}^d \mu_{j\ell} =1,\forall j\in J, \label{eqq:l1-9}\\
& \eta^+_{j\ell}, \eta^-_{j\ell}, \theta_{j\ell} \in \R^d_+, \forall j\in J,\ell=1,\ldots,d,\\
& \mu_{j\ell}, \xi_{j\ell}\in \{0,1\}, \forall j\in J,\ell=1,\ldots,d,
\end{align}
where $M_{ij}$ and $U_{j\ell}$ are big enough constants.

We have introduced in the above formulation some new variables to model the $\ell_\infty$-distance in the denominator of the residual \eqref{res:l1}. In particular, for each $j\in J$, the $d$-dimensional variable ${\boldsymbol \theta}_j$ models the vector $(|\bbeta_{j1}|, \ldots, |\bbeta_{jd}|)$ for which the maximum has to be taken; $\eta_{j\ell}^+$ represents $\max\{\bbeta_{j\ell} , 0\}$ and $\eta_{j\ell}^-$ the amount $\max\{-\bbeta_{j\ell} , 0\}$, for all $\ell =1, \ldots, d$. Clearly, one has that $\bbeta_j={\boldsymbol\eta}^+_j-{\boldsymbol\eta}_j^-$ and ${\boldsymbol \theta}_j= {\boldsymbol\eta}^+_j+{\boldsymbol\eta}_j^-$ as imposed in constraints \eqref{eqq:l1-3}-\eqref{eqq:l1-6}, where the auxiliary variables $\xi$ enforce that for each coordinate, either the positive or the negative part assumes value zero (avoiding other types of decompositions). Constraints \eqref{eqq:l1-7}, \eqref{eqq:l1-8} and \eqref{eqq:l1-9}  assure that $\max_{j\in J} |\bbeta_j| =1$ via the auxiliary binary variables $\mu_{j\ell}\in \{0,1\}$ that take value $1$ in exactly one position (the one where the maximum is achieved).

Thus, the formulation assures that $\max_{l=1, \ldots d} |\bbeta_{j\ell}| =1$, and then, the expression of the residual becomes $ \varepsilon_{x_i}(\bbeta_j,\alpha_j)  = \dmin_{j\in J}  |\alpha_j+\bbeta_{j}^t x_i|$ for all $i\in I$ and $j\in J$.

In this case, also \eqref{FHBEP} becomes a Mixed Integer Linear Programming problem.
\end{remark}

\section{Set Partitioning formulation\label{sec:4}}

In this section we alternatively reformulate \eqref{mofhp0} as a set partitioning problem (SPP) (see e.g., \cite{BP76}). Our SPP  is based on the idea that once the $p$ clusters of demand points are known , \eqref{mofhp0} reduces to finding the optimal hyperplanes for each of those clusters in which all the residuals are aggregated by means of an ordered median function.  In particular, let $S$ be a cluster of observations $S\subseteq I$. We denote by $e_S$ the cost of cluster $S$, i.e. the overall aggregation of the residuals of the data in $S$, and for each $i\in S$, $e^i_S$ the marginal contribution of observation $i$ in the cluster ($e_S = \dsum_{i \in S} e^i_S$). Finally, we define the variables
$$
y_S=\left\{ \begin{array}{ll} 1 & \mbox{if cluster $S$ is selected,}\\
0 & \mbox{otherwise} \end{array} \right., \quad \text{ for all } S \subseteq I.
$$
The set partitioning formulation for \eqref{mofhp0} is:
\begin{align}
\min & \dsum_{i\in I} \dsum_{S \subseteq I: S\ni i}  \lambda_i e^{(i)}_{S} y_S\\
\mbox{s.t. } & \sum_{S \subseteq I} y_S=p,  \label{c0:13}\\
 & \sum_{S \subset I:\atop S \ni i}  y_S=1, \forall \; i\in I,  \label{c0:14}\\
 & y_S\in \{0,1\}, \; \forall S\subseteq I.
\end{align}
where $e^{(i)}_S$ is the $i$-th element in the sorted sequence of (active) residuals. In the above formulation the objective function computes the ordered median aggregation of the residuals (each demand point $i$ allocated to its cluster $S$). Constraint \eqref{c0:13} assures that $p$ clusters have to be computed and constraints \eqref{c0:14} that each observation belongs to a single cluster.

In the same manner that we formulate the ordered median objective function in the compact formulation we can equivalently reformulate the problem above as follows:
\begin{align}
\min & \dsum_{k\in I} u_k+\dsum_{i \in I} v_i \\
\mbox{s.t. } &\;\; u_k + v_i \geq \lambda_k \;\dsum_{S \ni k} e^i_S y_S, \;\; \forall i, k\in I,\label{c:mp1}\\
 & \sum_{S \subseteq I} y_S=p,  \label{c:mp2}\\
 &  \sum_{S \subset I:\atop S \ni i}   y_S=1,\;\; \forall \; i\in I,  \label{c:mp3}\\
 & y_S\in \{0,1\}, \;\; \forall S\subset I,\nonumber\\
 & u_k, v_i \in \R,\;\; \forall i, k\in I.\nonumber
\end{align}
This problem will be referred to as the \textit{Master Problem}.

The problem above, although adequately solves the problem of finding the $p$ hyperplanes once the optimal clusters are computed, has an exponential number of variables (and coefficients to incorporate to constraints \eqref{c:mp1}), and then it is hard to solve unless the number of points is very small. Thus, we propose a column generation (CG) approach for solving, efficiently, the problem above by adding new variables to the model as needed and not considering all of them at the same time. A pseudocode indicating the general procedure is shown in Algorithm \ref{alg:cg}.

Initially, a (small) subset of the $y$-variables is considered (those indexed by the sets in $\mathcal{S}_0$) and a relaxed version of the problem is solved with only these variables. It implies to compute the amounts $e_S^i$ for all $S\in \mathcal{S}_0$ and $i\in S$. Next, it has to be checked whether the optimality condition is satisfied. If it is not the case, a new set of variables is found and added to the relaxed problem and the procedure is repeated.

\begin{algorithm}[!ht]
  \KwData{$\{x_1, \ldots, x_n\} \subseteq \R^d$, $p\in \Z_+$ ($p>0$), $\lambda_1 \geq \cdots \geq \lambda_n \geq 0$.}

\begin{description}
\item[1. Preprocessing:] Compute a set of initial solutions for the problem $\mathcal{S}_0 = \{S_1, \ldots, S_K\}$ with $S_k \subseteq I$ for all $k=1, \ldots, K$.
\item[2. Relaxed Master:] Solve the relaxed master problem:
\begin{align}
\min & \dsum_{k\in I} u_k+\dsum_{i \in I} v_i \nonumber\\
\mbox{s.t. } & u_k + v_i \geq \lambda_k \;\dsum_{S \ni k} e^i_S y_S, \forall i, k\in I,\nonumber\\
 & \sum_{S \in \mathcal{S}_0} y_S=p, \label{mp}\tag{RMP}\\
 &  \sum_{S \in \mathcal{S}_0:\atop S \ni i}   y_S=1,  \forall \; i\in I, \nonumber\\
 & 0\le y_S \le 1, \; S\in \mathcal{S}_0.\nonumber
\end{align}
\item[3. New Columns]: Check if new columns have to be added to \eqref{mp}.
\end{description}

\hspace*{0.4cm} \begin{minipage}[t]{0.4\textwidth}
\eIf{Optimality is satisfied}{$\mathcal{C}^*=\{S \subseteq I: y_S^*=1\}.$} 
 {Update $\mathcal{S}_0$ with the new columns and go to {\bf 2.}}
 \end{minipage}

\KwResult{\noindent$\{\mathcal{H}(\bbeta_S,\alpha_S): S \in \mathcal{C}^*\}$.}

 \caption{General Scheme for the CG approach.\label{alg:cg}}
\end{algorithm}

The crucial steps in the implementation of the CG approach are the following:
\begin{enumerate}
\item {\it Preprocessing:} Generation of initial feasible solutions in the form of initial clusters (and their costs). This step may be improved by the theoretical properties verified by the corresponding optimal hyperplanes. We have implemented different initial solutions based on properties of the optimal solution of median and center hyperplanes (see Section \ref{ss:mch}).
\item {\it Pricing:} As already mentioned,  in set partitioning problems, instead of solving initially the  problem with the whole set of exponentially many variables, the variables have to be incorporated \textit{on-the-fly} by solving adequate pricing subproblems derived from previously computed solutions until the optimality of the solution is guaranteed.
\item {\it Branching:} The rule that creates new nodes of the branch and bound tree when a fractional solution is found at a node of the search tree. In this problem, we have adapted the Ryan-and-Foster branching complemented by a secondary ad-hoc branching in some special situations.
\end{enumerate}

In what follows we describe how each of the items above is performed in our proposal.

\subsection{Preprocessing}

In the preprocesing phase, we generate different types of initial solutions, which implies the initialization of the CG algorithm with a given set of variables.

We consider different types of initial solutions derived from the construction of hyperplanes fitting the sets of points. First, to initialize the pool of columns, $\mathcal{S}_0$, we randomly generate hyperplanes passing through $d$ original points. Among the various strategies compared, we have eventually implemented one that performs completions with $d-2$ points of all possible couples of original points. This strategy augment $\frac{n(n-1)}{2}$ new variables into the pool. In addition, we also augment to the pool the best single hyperplane that fits all the points, assuring that the problem is feasible at the root node of the branch-and-price tree. Finally, apart from the above initial columns, we also charge an initial heuristic solution (in the $y$-variables) so as to have a good upper bound in our branch-and-price algorithm. Our algorithm chooses at random $p$ mutually disjoint subsets of  $d$ points and finds the hyperplanes determined by those $p$ sets of $d$ points. Next, we perform a 1-interchange heuristic generating a new hyperplane that replaces, one at a time,  one of those currently considered in the configuration until the first iteration where no improvement is possible. The incumbent set of hyperplanes and their corresponding allocations is considered an initial solution that is loaded into the solver.

\subsection{Median and center optimal hyperplanes\label{ss:mch}}

We have used the following properties to build the initial solutions of our CG approach since they determine optimal hyperplanes for specific objective functions, see e.g., \cite{Schobel03}.

\begin{lemma}
The following properties are verified:
\begin{enumerate}
\item {\rm Weak incidence property}: There exists an optimal median hyperplane passing through $d$ affinely independent points.
\item {\rm Pseudo-halving property}: Every optimal median hyperplane, $\mathcal{H(\beta^*,\alpha^*)}$  verifies
\begin{center}
$\# \Big\{i \in I: x_i\in\mathcal{H}^-(\beta^*,\alpha^*)\Big\} \leq \frac{n}{2}$ and  $\# \Big\{i \in I: x_i\in\mathcal{H}^+(\beta^*,\alpha^*) \Big\} \leq \frac{n}{2} $.
\end{center}
\item {\rm Weak blockedness property}: There exists an optimal center hyperplane that is at maximum distance from $d+1$ of the points.
\item {\rm Parallel facets property}: There exists an optimal center hyperplane that is parallel to a facet of the convex hull of the given points.
\end{enumerate}
\end{lemma}

%
For the more general ordered median objective function, we have proved the following result that characterizes the ordered median hyperplanes.  In what follows, we derive a novel result for these hyperplanes that will be useful in the preprocessing phase of our CG approach.

Let us introduce the following notation:

\begin{itemize}
\item Let $\mathcal{B}$ be the subdivision of $\mathbb{R}^{d+1}$ induced by the following arrangement of hyperplanes:
$$
B_{ij}^{ab} =\Big\{ (\beta,\alpha)\in \mathbb{R}^{d+1}: a(\beta x_i+\alpha)=b(\beta x_j+\alpha)\Big\}, \forall i,j\in I, a,b\in\{-1,1\}.
$$
\item  Let $\mathcal{S}$ be the subdivision of of $\mathbb{R}^{d+1}$ induced by the following arrangement of hyperplanes
$$
S_i=\Big\{(\beta,\alpha)\in \mathbb{R}^{d+1}: \beta x_i+\alpha =0\Big\}, \forall i \in I.$$
\end{itemize}

\begin{lemma}
If $\mathcal{H}(\beta,\alpha)$ is an optimal ordered median hyperplane then $(\beta,\alpha)$ is an extreme point of a cell in the subdivision of $\mathbb{R}^{d+1}$ induced by the intersection $\mathcal{B}\cap \mathcal{S}$.
\end{lemma}
\begin{proof}
For a given hyperplane $\mathcal{H}(\beta,\alpha)$, let us consider the objective function of the problem, namely $\sum_{i\in I} \lambda_i e_{(i)}$, where $e_i=D(\mathcal{H}(\beta,\alpha),x_i)$.

It is clear that within each cell of the subdivision $\mathcal{B}$ the sorting of the residuals does not change. In addition, in each cell of the subdivision $\mathcal{S}$ the sign of $\beta x_i+\alpha$ is either positive or negative (but does not change) for each $i\in I$. Therefore, if $C\in \mathcal{B}\cap \mathcal{S}$ is a cell in the subdivision induced by $\mathcal{B}\cap \mathcal{S}$, there is permutation $\sigma$ that fixes the sorting of the residuals and also a constant vector $(sign(\beta x_1+\alpha), \ldots, sign(\beta x_n+\alpha))\in \{-1,1\}^n$ such that
$$\sum_{i\in I} \lambda_i e_{(i)} = \sum_{i\in I} \lambda_i \frac{sign(\beta x_i+\alpha) (\beta x_i+\alpha)}{\|\beta\|^*}= \frac{\sum_{i\in I} \lambda_i sign(\beta x_i+\alpha) (\beta x_i+\alpha)}{\|\beta\|^*}.$$
The above function is the ratio of a non-negative linear function and a convex function, then it is quasiconcave provided that $(\beta,\alpha)\in C$. Therefore, it attains its minima at the extreme points of this region. Hence, if $\mathcal{H}(\beta,\alpha)$ is an optimal ordered median hyperplane $(\beta,\alpha)$ must be an extreme point of some of those cells.
\end{proof}
The above result allows us to interpret optimal ordered median hyperplanes also in terms of a geometrical description as those that meet $d$ conditions between the following cases: i) passing through points $x_i,\; i\in I$, and ii) being at the same distance of two points $x_i,x_j, \; i,j\in I$. Optimal ordered median hyperplanes must  also satisfy, for some $k=1,\ldots,d$, the following property: it contains $k$ points $x_i,\; i\in I$ and it is at the same distance from $d-k$ pairs $x_i, x_j$ $i,j\in I$.

In our computational results we have computed the initial solutions and the initial pool of variables for the objective functions of type median, $k$-centrum and centdian, using the weak incidence property, whereas for the center objective function we use the weak blockedness property.

\subsection{Pricing problem}

Certifying optimality in a CG approach avoiding the inclusion of all the columns into the relaxed master problem, \eqref{mp}, requires testing whether a new tentative column must be added to the problem. In case no new candidates are added to the master problem, the optimality is guaranteed, otherwise, one should add the new columns and repeat the process {(Step 3 in Algorithm \ref{alg:cg}). Searching for new columns to be added to the model will be performed by looking at the dual formulation of  the set partitioning formulation.

Let $\gamma$ be the dual variable for constraint \eqref{c:mp2}, $\phi_i$ the dual variables for constraints \eqref{c:mp3} and $\delta_{ik}$ the dual variables for constraints \eqref{c:mp1}. Then, the dual of the Master Problem is the following:
\begin{align*}
\max & -p \gamma+ \sum_{i\in I}^n \phi_i \\
\mbox{s.t. } & \sum_{k\in I} \delta_{ik} = 1, \forall i\in I, \\
& \sum_{i\in I} \delta_{ik} = 1,\forall k \in I, \\
& -\sum_{i\in S}\sum_{k=1}^n \lambda_ke_S^i\delta_{ik} - \gamma + \sum_{i\in S} \phi_i \leq 0,  \forall S \subseteq \mathcal{S}_0, \\
& \delta_{ik}, \gamma,\phi_i \geq 0.
\end{align*}
Hence, for any $S \subset \mathcal{S}_0$, since $y_S$ does not appear in the objective function, the reduced cost for variable $y_S$ is:
$$
\bar e_S=\gamma-\sum_{i\in S} \phi_i +\dsum_{i \in S} \sum_{k=1}^n \lambda_k \delta_{ik} e^i_S.
$$
Then, given an optimal dual solution $( \gamma^*,\phi^*, \delta^*)$, and considering the binary variables
$$
w_{i} = \left\{\begin{array}{cl}
1 & \mbox{ if the $i$th observation is chosen for the set $S$ indexing the new column,}\\
0 & \mbox{otherwise,}
\end{array}\right.
$$
the pricing problem is to choose the subset $S$ with minimum reduced cost, i.e., to solve:
\begin{align*}
\displaystyle\min&\displaystyle-\sum_{i\in I}\phi_i^*w_{i}+\gamma^*+\sum_{i\in I} c_i^* r_{i}&\\
\mbox{s.t. } & z_i \ge  \varepsilon_{x_i}(\bbeta,\alpha)), \;\; \forall i \in I,\\
&r_{i} \geq z_i - M(1-w_{i}) ,\;\; \forall i \in I,\\
&w_{i}\in\{0,1\},\;\;  \forall i\in I,\\
&z_i,r_{i}\ge 0, \;\; \forall i\in I,\\
&\bbeta \in \R^d, \alpha \in \R.
\end{align*}
where $c^*_i=\sum_{k=1}^n\lambda_{k}\delta^*_{ik}$, $\forall i \in I$.

If the optimal value of this problem is negative, the new column $y_{\hat S}$ is added to the pool, where $\hat S= \{ i : w_i = 1 \}$,  since its reduced cost in the (RMP) is negative, and thus,  it  improves the objective function of the master problem. Otherwise, optimality is certified and we are finished.

\subsubsection{Heuristic pricing}

The exact pricing routine described above is an NP-hard problem and thus in general, it takes time finding new columns to be added to the pool or to certify optimality of the reduced master problem. This last task cannot be avoided, provided that we design an exact solution algorithm. Nevertheless, in many occasions finding promising new variables can be done at very low computational time resorting to heuristic schemes.

In our problem, we propose to test hyperplanes chosen from a discrete set of potential candidates. To do so, we set a $d$-dimensional grid on the normalized space of $\alpha$ and $\bbeta$ coefficients. Each point represents a hyperplane to be tested. Once the candidate $(\alpha,\bbeta)$ is chosen we determine which set of points $S$ is going to be added to the new variable $y_S$. This is done with a simple greedy rule: choose those points with negative reduced cost with respect to $H_\beta$.

If after this process we find  a hyperplane that produces a negative reduced cost, we add this new column to the pool. Otherwise, we proceed with the exact pricer. This scheme speeds up the search for new columns without loosing the exactness of the whole algorithm.

\subsection{Branching}

The set partitioning formulation of  the  \ref{mofhp0} is often not solved at the root node, in contrast with what is stated in \cite{PJKW17}. Thus, some branching strategy must be implemented to cope with the branch and bound search. Ryan-Foster (R-F) is one of the most popular techniques for branching in set partitioning problems (see \cite{Ryan1981}). If a fractional solution is reached at a node, R-F creates two new branches as follows: Given to elements $i_1,i_2\in I$, they may never go together on a set in the whole branch, or they may always go together, i.e., if one of them belongs to a set $S$, the other one must also be included in $S$.

To implement this branching, we can take advantage of the $w_i$ variables defined on the previous section for the pricing subproblem, to easily adapt this way of branching in our problem, by means of the following constraints:
\begin{itemize}
\item[A)] $w_{i_1}+w_{i_2}=1$ ensuring that elements $i_1$ and $i_2$ are not assigned to the same hyperplane.
\item[B)] $w_{i_1}=w_{i_2}$ ensuring that elements $i_1$ and $i_2$ are assigned to the same hyperplane.
\end{itemize}

Moreover, in our formulation there is a new case in which, despite the fact of having fractional solutions on a node, we will not create new branches following the R-F rule. This fact is motivated because in our problem may appear different columns (different $y$-variables) but being associated to the same set $S$, although possibly with different hyperplanes.

Let $S\subseteq I$ be a subset of points and let $y_S^1,...,y_S^q$ be fractional variables for the same set $S$ although with different hyperplanes $\mathcal{H}(\beta^i,\alpha^i)$, $i=1\ldots q$, with $q > 1$, such that $\dsum_{i=1}^q y_S^i = 1$. If there are no more fractional variables, or the rest of the fractional variables of the node satisfy the same conditions for some other subsets of  points, we cannot apply R-F rule and either the node need not be branched (see Theorem \ref{th:RF} and Remark \ref{re:RF}) or a different branching strategy must be implemented in these cases.


Without loss of generality, we will describe the new branching for the case in which two fractional variables, $y_S^{1}$ and $y_S^{2}$, with hyperplanes $\mathcal{H}(\beta^1,\alpha^1)$ and $\mathcal{H}(\beta^2,\alpha^2)$, are obtained in a node for the same subset $S$.
In this situation, the new branching rule that we propose creates three new branches as follows:
\begin{itemize}
\item[1.] A branch where $y_S^{1} =1$ meaning that this variable will be in the solution in this branch. This is easily implemented in the pricing routine since it amounts to avoid considering the elements in $S$ in any further column in that branch because they are already in the set $S$ which is part of the solution. Therefore, it suffices to fix the variables $w_i=0, \ \forall i \in S$ in all the subproblems in the branch.
\item[2.] Analogously, it creates another branch where $y_S^{2} =1$.
\item[3.] The third branch sets $y_S^{1}=y_S^{2}=0$. This branch represents the case in which none of the original fractional solutions are part of the integer solution. Once again, this can be enforced by adding the following constraints to the pricing subproblems of the branch:
\begin{center}
$\displaystyle \left(\left(
|S|-\sum_{i\in S} w_i
\right) +
\left|
|S|-\sum_{i\in I} w_i
\right|
 + \sum_{\ell=1}^d |\beta_{\ell}^j-\beta_{\ell}^*| + |\alpha^j-\alpha^*|
  \right)\cdot M \geq 1, \ \ j=1,2, $
\end{center}
for a big enough constant $M$, where $\beta^*$ and $\alpha^*$ define the new hyperplane $\mathcal{H}(\beta^*,\alpha^*)$. These constraints will make the problem infeasible if and only if all the individuals in $S$, and only the individuals of $S$, belong to the new solution, and moreover, the new solution provides a hyperplane $\mathcal{H}(\beta^*,\alpha^*)$ that is equal to $\mathcal{H}(\beta^1,\alpha^1)$ or $\mathcal{H}(\beta^2,\alpha^2)$.

\end{itemize}

The alternative branching may be necessary in case of using general norm based residuals. Nevertheless, as we show below, the situation is simpler using vertical distance residuals.

\begin{theo} \label{th:RF} Ryan and Foster branching is enough in the set partitioning formulation of \eqref{mofhp0} for the vertical distance residuals: If for a subset of points $S \subseteq I$, there exists a fractional solution $0 < y_S < 1$ with $\# \left\lbrace y_S \neq 0 \right\rbrace > 1$, at a node of the branch and bound tree  then there exists an explicit solution that  combines these variables to obtain a single one satisfying $y_S=1$ ($\# \left\lbrace y_S \neq 0 \right\rbrace = 1$).
\end{theo}

\begin{proof}
Let us consider a subset of points $S\subseteq I$. At a fractional node, we can have two possible scenarios: 1) $\# \left\lbrace y_S \neq 0 \right\rbrace = 1$, hence, it would exist a single hyperplane (a facility) $\mathcal{H}(\beta,\alpha)$ that would serve the points of S, and hence, R-F branching is enough, and 2) $\# \left\lbrace y_S \neq 0 \right\rbrace > 1$. This latter case needs a further analysis since it may seem as if more than one facility would need to be involved to optimally serve the points in S.

Without loss of generality we can assume $\# \left\lbrace y_S \neq 0 \right\rbrace = 2$ (a case with $\# \left\lbrace y_S \neq 0 \right\rbrace > 2$ can be treated sequentially by smaller problems with two solutions). In this situation there are two variables, $y_S^1$ and $y_S^2$, with values $\sigma$ and $1-\sigma$, $\sigma \in (0,1)$, so that $y_S^1+y_S^2=1$.
These variables are represented by two hyperplanes $\mathcal{H}(\beta^1,\alpha^1)$ and $\mathcal{H}(\beta^2,\alpha^2)$, where the cost of a point $i \in S$ with coordinates $x\in\mathbb{R}^d$, $e_i$, is given by $e_i = \sigma D(x,\mathcal{H}(\beta^1,\alpha^1)) + (1-\sigma)D(x,\mathcal{H}(\beta^2,\alpha^2))$. We prove that the hyperplane $\mathcal{H}(\beta^*,\alpha^*)$ defined as
\begin{center}$\mathcal{H}(\beta^*,\alpha^*) = \left\lbrace z \in \mathbb{R}^d \ : \   \sigma (\alpha^1 + \beta^1z) + (1-\sigma)(\alpha^2 + \beta^2z) = 0 \right\rbrace$,
\end{center}
satisfies that $D(x,\mathcal{H}(\beta^*,\alpha^*))\leq \sigma D(x,\mathcal{H}(\beta^1,\alpha^1)) + (1-\sigma)D(x,\mathcal{H}(\beta^2,\alpha^2))$ for vertical  distance residuals, and this would mean that there exists a unique hyperplane that optimally serves all the points in $S$. Therefore, considering $y_S^*$, no further branching is required.

If we consider the normalized hyperplanes $\mathcal{H}(\beta^1,\alpha^1)$, and $\mathcal{H}(\beta^2,\alpha^2)$, such that $\beta^1_d =\beta^2_d=-1$, then $\beta^*_d=-1$, the vertical distance from x to $\mathcal{H}(\beta^*,\alpha^*)$ is
\begin{center}
$D_v(x,\mathcal{H}(\beta^*,\alpha^*) )= \left|x_d - \alpha^* - \dsum_{\ell=1}^{d-1}\beta^*_{\ell}x_{\ell}\right| $.
\end{center}
Hence,
\begin{align*}
D_v(x,\mathcal{H}(\beta^*,\alpha^*) )& = \left|x_d - \alpha^* - \dsum_{\ell=1}^{d-1}\beta^*_{\ell}x_{\ell}\right| \\
 &= \left|\sigma x_d + (1-\sigma)x_d  - (\sigma\alpha^1 + (1-\sigma)\alpha^2) - \dsum_{\ell=1}^{d-1}(\sigma\beta^1_{\ell}+(1-\sigma)\beta^2_{\ell})x_{\ell}\right| \\
  & \leq \sigma \left|x_d - \alpha^1 - \dsum_{\ell=1}^{d-1}\beta^1_{\ell}x_{\ell}\right| + (1-\sigma)\left|x_d - \alpha^2 - \dsum_{\ell=1}^{d-1}\beta^2_{\ell}x_{\ell}\right| \\
 & =\sigma D_v(x,\mathcal{H}(\beta^1,\alpha^1) ) + (1-\sigma)D_v(x,\mathcal{H}(\beta^2,\alpha^2) ).
\end{align*}
\end{proof}

\begin{remark} \label{re:RF}
We prove that under mild conditions, R-F branching is also enough for the $\ell_1$-norm based residuals.

Without loss of generality, assume that there is a solution with two fractional variables $y_S^1$ and $y_S^2$, with values $\sigma$ and $(1-\sigma)$, and corresponding hyperplanes $\mathcal{H}(\beta^1,\alpha^1)$ and $\mathcal{H}(\beta^2,\alpha^2)$. Let us define the set $SP = \left\lbrace j : |\beta_{j}^1| = |\beta_{j}^2|=1, j =1,\ldots,d \right\rbrace$. Hence, if $SP\neq\emptyset$, RF-branching is enough for the $\ell_1$-norm residuals.

Indeed, let $\hat \j$ be an index such that $|\beta_{\hat \j}^1|=|\beta_{\hat \j}^2|=1$ and define
$$sign(\hat \j)= \left\{ \begin{array}{ll} 1 & \mbox{ if } \beta_{\hat \j}^1\cdot \beta_{\hat \j}^2=1\\
-1 & \mbox{ if } \beta_{\hat \j}^1\cdot \beta_{\hat \j}^2=-1.
\end{array} \right.
$$
It is clear that for any $\hat \j\in SP$, $\beta^*_{\hat \j}=\alpha \beta^1+\sign(\hat \j) (1-\alpha) \beta^2$ satisfies $\|\beta^*_{\hat \j}\|_\infty=1$.
Consider for any $\hat \j\in SP$ the hyperplane
$$\mathcal{H}(\beta^*_{\hat \j},\alpha^*) = \left\lbrace z \in \mathbb{R}^d \ : \   \sigma (\alpha^1 + \beta^1z) +\sign(\hat \j)  (1-\sigma)(\alpha^2 + \beta^2z) = 0 \right\rbrace,$$
then for any individual $i\in S$ with coordinates $x_i\in\mathbb{R}^d$, taking into account that $||\beta^1||_\infty =||\beta^2||_\infty =1$, we obtain that
\begin{align*}
\displaystyle D_{\ell_1}(x_i,\mathcal{H}(\beta^*,\alpha^*)) &=
\frac{|\alpha^* + \sum_{\ell=1}^d\beta^*_\ell x_{i\ell}|}{||\beta^*||_\infty}\\
 & \leq \sigma\frac{|\alpha^1 + \sum_{\ell=1}^d\beta^1_\ell x_{i\ell}|}{||\beta^1||_\infty} +
(1-\sigma)\frac{|\alpha^2 + \sum_{\ell=1}^d\beta^2_\ell x_{i\ell}|}{||\beta^2||_\infty}\\
 & = \sigma D_{\ell_1}(x_i,\mathcal{H}(\beta^1,\alpha^1)) +(1-\sigma)D_{\ell_1}(x_i,\mathcal{H}(\beta^2,\alpha^2)).
\end{align*}
 and hence, $y_S^*=1$ is an optimal solution for the problem.

\end{remark}

\section{Computational Results\label{sec:comp_result}}

A series of computational experiments has been performed in order to test the two proposed methodologies. We consider two different sets of instances, one based on \cite{EWC74} dataset and another on synthetic data. For all of them we solve \eqref{mofhp0} for four different objective functions: Weber (W), Center (C), $\lceil \frac{n}{2}\rceil$-Centrum (K) ($\lambda=(\overbrace{1, \ldots, 1}^{\lceil \frac{n}{2}\rceil}, 0, \ldots, 0)$) and $0.9$-Centdian (D) ($\lambda=(1, 0.9, \ldots, 0.9)$) and with the two proposed approaches: the compact approach based on formulation \eqref{FHBEP} and with the branch-and-price methodology.  We test the performance of the algorithms on two different types of residuals: $\ell_1$-norm based residuals and  absolute value vertical distance residuals.

The models were coded in C and solved with SCIP v.6.0.1 using as optimization solver CPLEX 12.8 in a Mac OS El Capitan with a Core i7 CPU clocked at 2.8 GHz and 16GB of RAM memory. A time limit of $5$ hours was fixed for all the instances. It is well-known in the field of location analysis that continuous multifacility ordered median problems are very difficult to solve and already  problems of moderate sizes ($n=50$ demand points) can not often be solved to optimality (see e.g., \cite{BPE16}). The same or even a harder behavior should be expected here since these problems introduce a new degree of difficulty in the representation of general distance based residuals.

\subsection{\cite{EWC74} dataset}

First, we tested our approach on instances based on the classical planar $50$-points dataset provided by \cite{EWC74}. We randomly generate five instances from such a dataset  with sizes $n \in \{ 20, 30, 40, 45\}$  and the entire complete original instance with $n=50$. We run the models for $p\in \{2,5\}$ hyperplanes. The average results obtained for these instances are shown in tables \ref{tab:EilonV} and  \ref{tab:EilonL1}. There, for each combination of $n$ (size of the instance), $p$ (number of hyperplanes to be located) and \texttt{type} (ordered median objective function to be minimized), we provide both for the compact formulation \ref{FHBEP} (Compact) and for the branch-and-price (B\&P) approach: the CPU time in seconds needed to solve the problem (\texttt{Time} (secs.)), the MIP Gap (\texttt{GAP}) remaining after the time limit, the number of nodes (\texttt{Nodes}) explored in the branch and bound tree and the RAM memory (\texttt{Memory} (MB)) in Megabytes required during the execution process. Within each column (\texttt{Time}, \texttt{GAP}, \texttt{Nodes} and \texttt{Memory}), we highlight in bold the best result between the two formulations, namely Compact or B\&P. Table \ref{tab:EilonV} gives the results for the models with vertical distance residuals while Table \ref{tab:EilonL1} provides the results for the $\ell_1$-norm residuals.

As expected,  the difficulty of the problem increases with $n$ and $p$. Problems with smaller $n$ are easier and $p=2$ is also easier than $p=5$. We also observe in Table \ref{tab:EilonV} that the B\&P approach is more efficient than formulation \ref{FHBEP} in all cases with the only exception of center problems. For that type of problem with vertical distance residuals the compact formulation is able to solve all instances in all cases whereas the B\&P reports an overall gap of $20.36\%$. As it can be expected the number of nodes to be explored in order to solve the problems is several orders of magnitude larger for the compact formulation than for the B\&P algorithm. This fact shows that the former formulation is much less accurate than the latter thus requiring many more nodes to be explored to solve the problems, implying a better scalability of the B\&P approach. In addition, \ref{FHBEP} requires very large RAM memory resources since already for $n=50$ points, it demands, in some cases, more than $11$ GB whereas B\&P solves the problems using at most $4$ GB of RAM memory.

Turning to Table \ref{tab:EilonL1} we observe, as expected, that using $\ell_1$-norm residuals make problems harder to solve mainly due to the representation of the projections point-to-hyperplane stated in Remark \ref{rem3}. In this case, the overall gaps increase from $29.36\%$ and $11.24\%$ in Table \ref{tab:EilonV}, for \ref{FHBEP} and B\&P, respectively, to $62.69\%$ and $20.53\%$. This behavior is more severe for  \ref{FHBEP} because already for $n=20$ and $p=5$ that formulation  is not able to certify optimality for any of the problems regardless of the type within the time limit. On the contrary, B\&P is affected less and its behavior is similar to what one observes for vertical distance in Table \ref{tab:EilonV}. The rest of comments regarding number of nodes and memory requirements are similar to those given previously for vertical distances.

\renewcommand{\arraystretch}{0.8}
 \begin{table}
  \centering{\small\begin{tabular}{|c|c|c||rr|rr|rr|rr|}\cline{4-11}
    \multicolumn{3}{c||}{} & \multicolumn{2}{c|}{\texttt{Time} (secs.)} &  \multicolumn{2}{c|}{\texttt{GAP}} &  \multicolumn{2}{c|}{\texttt{Nodes}}  &  \multicolumn{2}{c|}{\texttt{Memory} (MB)} \\\hline
   $n$ &$p$ & \texttt{type} & Compact & B\&P &  Compact & B\&P & Compact & B\&P &Compact & B\&P\\\hline
\multirow{8}{*}{20} & \multirow{4}{*}{2} &  W    &\bf 2.08  & 68.15 & 0.00\% & 0.00\% & 2878  &\bf 2     & \bf3     & 23 \\
      &       &  K    &\bf 2.10  & 69.18 & 0.00\% & 0.00\% & 2878  & \bf2     &\bf 3     & 23 \\
      &       &  D    & \bf6.70  & 86.14 & 0.00\% & 0.00\% & 2904  &\bf 2     &\bf 3     & 24 \\
      &       &  C    &\bf 0.11  & 3171.27 & 0.00\% & 0.00\% & \bf35    & 5070  & \bf1     & 1347 \\\cline{2-11}
      & \multirow{4}{*}{5} &  W    & 12411.90 & \bf23.48 & 18.67\% &\bf 0.00\% & 23623465 & \bf16    & 2226  & \bf13 \\
      &       &  K    & 12422.17 &\bf 23.31 & 18.70\% & \bf0.00\% & 23560539 & \bf16    & 2221  & \bf13 \\
      &       &  D    & 13082.16 &\bf 43.96 & 18.67\% &\bf 0.00\% & 24131841 & \bf26    & 2161  &\bf 15 \\
      &       &  C    & \bf103.70 & 2798.43 & 0.00\% & 0.00\% & 204693 &\bf 12259 & \bf36    & 300 \\\hline
\multicolumn{3}{|c}{\bf Average 20:} & 4753.86 & \bf785.49 & 7.00\% &\bf 0.00\% & 8941154 & \bf2174  & 832   &\bf 220 \\\hline
\multirow{8}{*}{30} & \multirow{4}{*}{2} &  W    &\bf 52.13 & 1439.90 & 0.00\% & 0.00\% & 60401 & \bf11    & \bf9     & 109 \\
      &       &  K    & \bf52.77 & 1440.09 & 0.00\% & 0.00\% & 60401 &\bf 11    &\bf 9     & 109 \\
      &       &  D    & \bf57.57 & 2410.21 & 0.00\% & 0.00\% & 62889 &\bf 7     &\bf 9     & 107 \\
      &       &  C    &\bf 0.17  & 18000.00 & \bf0.00\% & 25.73\% & \bf40    & 2833  &\bf 3     & 4033 \\\cline{2-11}
      & \multirow{4}{*}{5} &  W    & 18000.00 & \bf654.68 & 87.13\% & \bf0.00\% & 22547601 & \bf109   & 7194  &\bf 47 \\
      &       &  K    & 18000.00 & \bf653.84 & 87.11\% &\bf 0.00\% & 22459376 & \bf109   & 7161  & \bf47 \\
      &       &  D    & 18000.00 &\bf 242.66 & 83.94\% & \bf0.00\% & 22252049 & \bf25    & 7240  & \bf41 \\
      &       &  C    & \bf349.26 & 11310.31 &\bf 0.00\% & 28.76\% & 503141 & \bf7537  & \bf89    & 1318 \\\hline
\multicolumn{3}{|c}{\bf Average 30:} & 6814.04 & \bf4518.97 & 32.27\% &\bf 6.81\% & 8493237 & \bf1330  & 2714  & \bf726 \\\hline
\multirow{8}{*}{40} & \multirow{4}{*}{2} &  W    & \bf1870.93 & 18000.00 &\bf 0.00\% & 11.71\% & 1453146 & \bf1     & \bf91    & 251 \\
      &       &  K    & \bf1923.60 & 18000.00 & \bf0.00\% & 11.73\% & 1453146 &\bf 1     & \bf91    & 248 \\
      &       &  D    & \bf1765.95 & 18000.00 &\bf 0.00\% & 10.56\% & 1290038 &\bf 1     & \bf87    & 244 \\
      &       &  C    & \bf0.26  & 17809.35 &\bf 0.00\% & 22.50\% &\bf 81    & 408   &\bf 4     & 1264 \\\cline{2-11}
      & \multirow{4}{*}{5} &  W    & 18000.00 & \bf15077.97 & 99.96\% & \bf1.88\% & 15197462 &\bf 280   & 9801  & \bf141 \\
      &       &  K    & 18000.00 & \bf15029.85 & 99.96\% & \bf1.77\% & 15210786 & \bf281   & 9792  & \bf142 \\
      &       &  D    & 18000.00 & \bf3346.72 & 99.83\% & \bf0.00\% & 14810951 &\bf 864   & 9700  & \bf183 \\
      &       &  C    & \bf982.10 & 12375.95 &\bf 0.00\% & 18.17\% & 1196810 & \bf1358  & \bf205   & 1559 \\\hline
\multicolumn{3}{|c}{\bf Average 40:} &\bf 7567.91 & 14705.75 & 37.47\% & \bf9.79\% & 6326552 & \bf399   & 3722  & \bf504 \\\hline
\multirow{8}{*}{45} & \multirow{4}{*}{2} &  W    & \bf10238.75 & 18000.00 & \bf0.00\% & 27.97\% & 6492828 & \bf1     & \bf219   & 351 \\
      &       &  K    & 10438.84 & \bf9514.96 &\bf 0.00\% & 5.05\% & 6532368 &\bf 9     &\bf 208   & 401 \\
      &       &  D    &\bf 10192.16 & 18000.00 & \bf0.00\% & 45.37\% & 6066819 &\bf 1     & \bf217   & 336 \\
      &       &  C    &\bf 0.35  & 14541.26 &\bf 0.00\% & 7.58\% & \bf133   & 2286  &\bf 5     & 932 \\\cline{2-11}
      & \multirow{4}{*}{5} &  W    & 18000.00 & 18000.00 & 99.86\% & \bf29.67\% & 12121664 & \bf32    & 10427 & \bf139 \\
      &       &  K    & 18000.00 & 18000.00 & 100.00\% &\bf 47.44\% & 12720196 & \bf25    & 11047 &\bf 142 \\
      &       &  D    & 18000.00 & \bf7525.10 & 99.46\% &\bf 0.01\% & 12193404 & \bf2383  & 9076  & \bf336 \\
      &       &  C    & \bf1268.05 & 17486.12 & \bf0.00\% & 36.94\% & 1570195 &\bf 1490  & \bf244   & 2245 \\\hline
\multicolumn{3}{|c}{\bf Average 45:} & \bf10767.32 & 15133.56 & 37.41\% & \bf25.00\% & 7212201 &\bf 778   & 3930  & \bf610 \\\hline
\multirow{8}{*}{50} & \multirow{4}{*}{2} &  W    & 18000.00 & 18000.00 & 22.46\% &\bf 3.48\% & 9401360 & \bf1     & 1593  & \bf582 \\  
      &       &  K    & 18000.00 & 18000.00 & 22.57\% & \bf3.48\% & 9275884 & \bf1     & 1584  & \bf583 \\
      &       &  D    & 18000.00 & 18000.00 & 21.06\% & \bf2.84\% & 8902849 & \bf1     & 1238  & \bf515 \\
      &       &  C    &\bf 0.29  & 18000.00 & \bf0.00\% & 19.79\% & \bf37    & 372   &\bf 6     & 923 \\\cline{2-11}
      & \multirow{4}{*}{5} &  W    & 18000.00 & 18000.00 & 100.00\% & \bf52.33\% & 10760962 & \bf1     & 11353 & \bf127 \\
      &       &  K    & 18000.00& 18000.00 & 100.00\% & \bf52.33\% & 10743398 &\bf 1     & 11335 & \bf126 \\
      &       &  D    & 18000.00 & 18000.00 & 100.00\% & \bf46.21\% & 9732814 &\bf 1     & 9864  & \bf134 \\
      &       &  C    &\bf 1778.36 & 18000.00 & \bf0.00\% & 43.86\% & 2234747 &\bf 470   &\bf 281   & 1432 \\\hline
\multicolumn{3}{|c}{\bf Average 50:} &\bf 13722.40 & 18000.00 & 45.76\% & \bf28.04\% & 7631506 &\bf 106   & 4657  & \bf553 \\\hline
\multicolumn{3}{|c}{\bf Total Average:} &\bf 7773.24 & 9224.75 & 29.36\% & \bf11.24\% & 7737963 &\bf 1120  & 2888  & \bf517 \\\hline
\end{tabular}%
   \caption{Results for \cite{EWC74} instances for vertical distance. \label{tab:EilonV}}}
\end{table}%

 \begin{table}
  \centering{\small\begin{tabular}{|c|c|c||rr|rr|rr|rr|}\cline{4-11}
    \multicolumn{3}{c||}{} & \multicolumn{2}{c|}{\texttt{Time} (secs.)} &  \multicolumn{2}{c|}{\texttt{GAP}} &  \multicolumn{2}{c|}{\texttt{Nodes}}  &  \multicolumn{2}{c|}{\texttt{Memory} (MB)} \\\hline
   $n$ &$p$ & \texttt{type} & Compact & B\&P &  Compact & B\&P & Compact & B\&P &Compact & B\&P\\\hline
\multirow{8}{*}{20} & \multirow{4}{*}{2} &  W    & 166.73 & \bf136.75 & 0.00\% & 0.00\% & 163750 & \bf21    & \bf17    & 30 \\
      &       &  K    & 167.49 & \bf136.65 & 0.00\% & 0.00\% & 163750 & \bf21    & \bf17    & 30 \\
      &       &  D    & 624.85 & \bf103.66 & 0.00\% & 0.00\% & 690721 & \bf26    & 40    & \bf30 \\
      &       &  C    & \bf0.98  & 10126.55 &\bf 0.00\% & 5.13\% & \bf1398  & 3597  & \bf3     & 2141 \\\cline{2-11}
      & \multirow{4}{*}{5} &  W    & 18000.00 & \bf111.13 & 100.00\% &\bf 0.00\% & 29829455 &\bf 32    & 10437 & \bf14 \\
      &       &  K    & 18000.00 &\bf 109.85 & 100.00\% &\bf 0.00\% & 30416596 & \bf32    & 10655 & \bf14 \\
      &       &  D    & 18000.00 &\bf 56.87 & 100.00\% &\bf 0.00\% & 31528234 & \bf15    & 10624 & \bf13 \\
      &       &  C    & 18000.00 &\bf 15315.35 & 100.00\% & \bf12.23\% & 40775405 & \bf18269 & 7513  & \bf2118 \\\hline
\multicolumn{3}{|c}{\bf Average 20:} & 9120.10 & \bf3262.10 & 50.00\% & \bf2.17\% & 16696164 & \bf2752  & 4913  &\bf 549 \\\hline
\multirow{8}{*}{30} & \multirow{4}{*}{2} &  W    & 13046.35 & \bf4509.86 & 28.75\% &\bf 0.00\% & 13086135 & \bf26    & 1187  &\bf 123 \\
      &       &  K    & 13034.52 & \bf4507.75 & 28.74\% &\bf 0.00\% & 13098248 & \bf26    & 1188  & \bf123 \\
      &       &  D    & 11959.18 &\bf 4595.89 & 26.93\% &\bf 0.01\% & 13057250 &\bf 27    & 1724  &\bf 127 \\
      &       &  C    & \bf2.92  & 12061.62 & \bf0.00\% & 19.61\% & 1192  & \bf507   & \bf4     & 2154 \\\cline{2-11}
      & \multirow{4}{*}{5} &  W    & 18000.00 &\bf 947.24 & 98.89\% &\bf 0.00\% & 20504433 & \bf35    & 10616 &\bf 39 \\
      &       &  K    & 18000.00 & \bf927.14 & 98.88\% &\bf 0.00\% & 21254042 &\bf 35    & 11006 & \bf39 \\
      &       &  D    & 18000.00 & \bf1885.54 & 100.00\% & \bf0.00\% & 20197549 &\bf 140   & 10972 & \bf49 \\
      &       &  C    & \bf14811.15 & 18000.00 & 80.00\% &\bf 46.40\% & 25951648 & \bf3028  & 8158  & \bf3031 \\\hline
\multicolumn{3}{|c}{\bf Average 30:} & 13356.84 & \bf5929.39 & 57.77\% &\bf 8.25\% & 15893812 & \bf478   & 5607  & \bf711 \\\hline
\multirow{8}{*}{40} & \multirow{4}{*}{2} &  W    & 18000.00 & 18000.00 & 42.82\% &\bf 6.40\% & 13047861 & \bf1     & 2218  & \bf201 \\
      &       &  K    & 18000.00 & 18000.00 & 42.95\% &\bf 6.40\% & 12998370 & \bf1     & 2214  & \bf201 \\
      &       &  D    & 18000.00 & 18000.00 & 65.74\% & \bf7.03\% & 10642421 & \bf1     & 1809  & \bf213 \\
      &       &  C    & \bf2.64  & 17184.72 &\bf 0.00\% & 39.56\% & 3792  & \bf124   & \bf5     & 1593 \\\cline{2-11}
      & \multirow{4}{*}{5} &  W    & 18000.00 & 18000.00 & 100.00\% & \bf39.89\% & 14778541 &\bf 49    & 10698 & \bf98 \\
      &       &  K    & 18000.00 & 18000.00 & 100.00\% & \bf39.20\% & 15280145 & \bf59    & 11070 & \bf101 \\
      &       &  D    & 18000.00 & 18000.00 & 100.00\% &\bf 35.47\% & 14043320 & \bf111   & 9495  & \bf145 \\
      &       &  C    & 18000.00 & 18000.00 & 100.00\% &\bf 60.84\% & 23716675 & \bf299   & 11125 & \bf1002 \\\hline
\multicolumn{3}{|c}{\bf Average 40:} & \bf15750.45 & 17900.37 & 68.94\% &\bf 29.35\% & 13063891 & \bf81    & 6079  & \bf444 \\\hline
\multirow{8}{*}{45} & \multirow{4}{*}{2} &  W    & 18000.00 & 18000.00 & 42.39\% & \bf4.85\% & 10512338 & \bf1     & 2795  & \bf287 \\
      &       &  K    & 18000.00& 18000.00 & 49.79\% &\bf 25.56\% & 10903011 & \bf1     & 2767  & \bf299 \\
      &       &  D    & 18000.00 & 18000.00 & 62.64\% & \bf24.59\% & 8683785 & \bf1     & 2263  & \bf296 \\
      &       &  C    & \bf2.13  & 18000.00 &\bf 0.00\% & 41.79\% & 2251  & \bf37    & \bf6     & 1036 \\\cline{2-11}
      & \multirow{4}{*}{5} &  W    & 18000.00 & 18000.00 & 100.00\% & \bf49.90\% & 11757058 &\bf 2     & 10826 &\bf 97 \\
      &       &  K    & 18000.00 & 18000.00 & 100.00\% & \bf55.13\% & 12704430 & \bf1     & 11911 &\bf 98 \\
      &       &  D    & 18000.00& 18000.00 & 100.00\% &\bf 48.34\% & 11591052 &\bf 3     & 9553  & \bf95 \\
      &       &  C    & 18000.00 & 18000.00 & 100.00\% &\bf 61.79\% & 22769758 & \bf133   & 11064 & \bf728 \\\hline
\multicolumn{3}{|c}{\bf Average 45:} & \bf15750.39 & 18000.00 & 69.35\% & \bf38.99\% & 11115460 & \bf22    & 6398  & \bf367 \\\hline
\multirow{8}{*}{50} & \multirow{4}{*}{2} &  W    & 18000.00 & 18000.00 & 96.53\% & \bf8.22\% & 7215656 & \bf1     & 3368  & \bf371 \\
      &       &  K    & 18000.00 & 18000.00 & 96.51\% &\bf 8.22\% & 7288856 &\bf 1     & 3402  & \bf371 \\
      &       &  D    & 18000.00 & 18000.00 & 96.34\% &\bf 8.21\% & 6792290 &\bf 1     & 2722  & \bf338 \\
      &       &  C    &\bf 4.21  & 18000.00 &\bf 0.00\% & 47.22\% & 5241  & \bf15    & \bf8     & 680 \\\cline{2-11}
      & \multirow{4}{*}{5} &  W    & 18000.00 & 18000.00 & 100.00\% & \bf53.68\% & 11063050 &\bf 1     & 9171  & \bf109 \\
      &       &  K    & 18000.00 & 18000.00 & 100.00\% & \bf53.68\% & 11115826 & \bf1     & 9212  &\bf 109 \\
      &       &  D    & 18000.00 & 18000.00 & 100.00\% &\bf 55.80\% & 11147925 & \bf1     & 9840  & \bf118 \\
      &       &  C    & 18000.00 & 18000.00 & 100.00\% & \bf63.66\% & 19632691 & \bf52    & 9468  & \bf488 \\\hline
\multicolumn{3}{|c}{\bf Average 50:} &\bf 15750.64 & 18000.00 & 86.17\% & \bf37.34\% & 9282692 &\bf 9     & 5899  & \bf323 \\\hline
\multicolumn{3}{|c}{\bf Total Average:} & 13601.88 & \bf11593.62 & 62.69\% & \bf20.53\% & 13958539 & \bf794   & 5756  & \bf508 \\\hline
\end{tabular}%

   \caption{Results for \cite{EWC74} instances for $\ell_1$-distance. \label{tab:EilonL1}}}
\end{table}%

\subsection{Synthetic Instances}

We have also randomly generated another set of instances to evaluate the performance of the two solution approaches depending on the space dimension ($d$). We have generated five instances of random points in the unit hypercube for each meaningful combination of $n\in \{20,30,40,45,50\}$, $p\in \{2,5,10\}$ and $d\in\{2,3,8\}$ (note that for these datasets, we have included additionally $p=10$ to analyze how increasing the number of hyperplanes affects the complexity for larger space dimension ($d=8$)). At this point, it is important to point out that several combinations of the above factors result in trivial problems, for instance for $n=20$ and $p=10$ there is always a solution passing through all the points and thus with zero objective value. All those cases that give rise to trivial solutions are not reported.
Table \ref{tab:SynthV} reports the results for the models with vertical distance residuals while Table \ref{tab:SynthL1} provides the results for the $\ell_1$-norm residuals. We report the same information as the one provided in the previous section but this time the results do not distinguish the type of objective function but the dimension of the space. (Needless to say that all the results disaggregated also by type are available upon request.)

For this dataset the results reinforce our previous observations in that  for problems with vertical distances (see Table \ref{tab:SynthV}), \ref{FHBEP} is much weaker than B\&P for $p=5,10$ and  in any dimension. In this case, however as seen in Table \ref{tab:SynthV} there are cases where for $p=2$ \ref{FHBEP} (see column \textit{Compact}) is more efficient. Turning to problems with $\ell_1$-norm residuals the performance is more homogeneous and B\&P is more efficient than \ref{FHBEP} for all $n$, $p$ and $d$. Once again, one observes that problems with $\ell_1$-norm residuals are more difficult than with vertical residuals. The overall gaps increase from $51.49\%$ and $29.93\%$ in Table \ref{tab:SynthV}, for \ref{FHBEP} and B\&P, respectively, to $83.78\%$ and $37.41\%$ in Table \ref{tab:SynthL1}.

 \begin{table}
  \centering{\small\begin{tabular}{|c|c|c||rr|rr|rr|rr|}\cline{4-11}
    \multicolumn{3}{c||}{} & \multicolumn{2}{c|}{\texttt{Time} (secs.)} &  \multicolumn{2}{c|}{\texttt{GAP}} &  \multicolumn{2}{c|}{\texttt{Nodes}}  &  \multicolumn{2}{c|}{\texttt{Memory} (MB)} \\\hline
   $n$ &$p$ & $d$ & Compact & B\&P &  Compact & B\&P & Compact & B\&P &Compact & B\&P\\\hline
\multirow{5}{*}{20} & \multirow{3}{*}{2} & {2} & \bf11.00 & 56.48 &0.00\% &0.00\% & 2710  &\bf 139   &  \bf 2    & 49 \\
      &       & {3} & \bf4.39  & 233.22 &0.00\% &0.00\% & 2922  &\bf 417   & \bf3     & 137 \\
      &       & {8} &\bf 30.88 & 1506.93 &0.00\% &0.00\% & 29777 & \bf1530  & \bf3     & 782 \\\cline{2-11}
      & \multirow{2}{*}{5} & {2} & 13017.91 & \bf3930.76 & 46.32\% & \bf1.25\% & 13909807 &\bf 40836 & 2209  & \bf430 \\
      &       & {3} & 18000.00 &\bf 4516.39 & 100.00\% & \bf19.56\% & 19586370 &\bf 8330  & 629   & \bf29 \\\hline
\multicolumn{3}{|c}{\bf Average 20: } & 6212.84 &\bf 2048.76 & 29.26\% & \bf4.16\% & 6706317 & \bf10250 & 569   & \bf285 \\\hline
\multirow{6}{*}{30} & \multirow{3}{*}{2} & {2} &\bf 55.66 & 2879.51 &\bf0.00\% & 1.34\% & 44038 & \bf1039  & \bf7     & 1218 \\
      &       & {3} & \bf60.81 & 8779.12 &\bf0.00\% & 5.77\% & 48533 & \bf1741  & \bf8     & 2737 \\
      &       & {8} & \bf414.28 & 18000.00 &\bf0.00\% & 67.19\% & 270055 &\bf 779   &   \bf23   & 1516 \\\cline{2-11}
      & \multirow{2}{*}{5} & {2} & 13933.75 &\bf 5046.83 & 74.14\% & \bf11.01\% & 11575613 &\bf 7989  & 4058  &\bf 1274 \\
      &       & {3} & 18000.00 & \bf12362.06 & 100.00\% & \bf18.32\% & 15098323 & \bf5033  & 3187  & \bf384 \\ \cline{2-11}
      & {10} & {2} & 18000.00 &\bf 4523.03 & 100.00\% & \bf11.05\% & 15536270 &\bf 10046 & 1572  & \bf298 \\\hline
\multicolumn{3}{|c}{\bf Average 30: } &\bf 8410.77 & 8598.55 & 45.69\% & \bf19.11\% & 7095472 & \bf4438  & 1476  & \bf1238 \\\hline
\multirow{7}{*}{40} & \multirow{3}{*}{2} & {2} & \bf1490.88 & 17404.04 &\bf0.00\% & 14.85\% & 903463 & \bf805   & \bf56    & 2186 \\
      &       & {3} & \bf1164.19 & 18000.00 &\bf0.00\% & 18.04\% & 726140 & \bf466   & \bf40    & 1579 \\
      &       & {8} &\bf 8455.38 & 18000.00 &\bf0.00\% & 71.44\% & 4005359 &\bf 59    &\bf 140      & 417 \\\cline{2-11}
      & \multirow{2}{*}{5} & {2} & 15809.48 & \bf12850.71 & 75.78\% &\bf 18.52\% & 10566235 &\bf 3642  & 5303  & \bf1187 \\
      &       & {3} & 18000.00 & 18000.00 & 100.00\% & \bf57.67\% & 12378840 & \bf705   & 5061  & \bf412 \\\cline{2-11} \cline{2-11}
      & \multirow{2}{*}{10} & {2} & 18000.00 &\bf 5498.89 & 100.00\% &\bf 19.66\% & 12196952 & \bf1770  & 2179  & \bf159 \\
      &       & {3} & 18000.00 & \bf13721.20 & 100.00\% & \bf43.81\% & 12359358 & \bf584   & 1121  & \bf79 \\\hline
\multicolumn{3}{|c}{\bf Average 40: } & \bf11560.01 & 14784.91 & 53.68\% & \bf34.85\% & 7590907 & \bf1147  & 1986  & \bf860 \\\hline
\multirow{8}{*}{45} & \multirow{3}{*}{2} & {2} &\bf 11045.65 & 18000.00 &\bf 4.56\% & 25.70\% & 5700797 & \bf587   & \bf316   & 2353 \\
      &       & {3} &\bf 8390.99 & 18000.00 &\bf0.00\% & 25.64\% & 4494533 & \bf235   & \bf127   & 1205 \\
      &       & {8} &\bf 13570.97 & 18000.00 & \bf33.32\% & 67.26\% & 5409179 &\bf 11    &     962  &\bf 324 \\\cline{2-11}
      & \multirow{3}{*}{5} & {2} & 16704.69 & \bf16218.37 & 75.04\% &\bf 33.86\% & 9858979 &\bf 2061  & 6104  &\bf 945 \\
      &       & {3} & 18000.00 & 18000.00 & 100.00\% &\bf 63.19\% & 11396830 & \bf353   & 5914  &\bf 344 \\
      &       & {8} & 18000.00 & 18000.00 & 100.00\% & 100.00\% & 11094838 &\bf 1     &   329    & \bf64 \\\cline{2-11}
      & \multirow{2}{*}{10} & {2} & 18000.00 & \bf6512.29 & 100.00\% & \bf20.37\% & 11236807 & \bf889   & 2124  &\bf 111 \\
      &       & {3} & 18000.00 &\bf 9421.82 & 100.00\% & \bf28.47\% & 10911247 &\bf 272   & 1471  &\bf 71 \\\hline
\multicolumn{3}{|c}{\bf Average 45: } & \bf15214.06 & 15307.51 & 64.12\% & \bf45.56\% & 8762901 &\bf 551   & 2168  & \bf677 \\\hline
\multirow{8}{*}{50} & \multirow{3}{*}{2} & {2} & \bf13500.09 & 18000.00 & 18.38\% &\bf 6.95\% & 6135466 &\bf 393   & \bf1132  & 2361 \\
      &       & {3} & \bf13500.51 & 18000.00 & \bf20.60\% & 30.71\% & 6182966 & \bf112   & \bf989   & 1209 \\
      &       & {8} & \bf13568.18 & 18000.00 & \bf45.65\% & 67.16\% & 4649063 & \bf2     &    1243   & \bf407 \\\cline{2-11}
      & \multirow{3}{*}{5} & {2} & \bf16593.29 & 18000.00 & 75.02\% & \bf51.32\% & 8984044 &\bf 801   & 7563  & \bf729 \\
      &       & {3} & 18000.00 & 18000.00 & 100.00\% & \bf65.74\% & 10380563 & \bf157   & 5969  &\bf 268 \\
      &       & {8} & 18000.00 & 18000.00 & 100.00\% & 100.00\% & 10407264 & \bf1     & 1472  &\bf 83 \\\cline{2-11}
      & \multirow{2}{*}{10} & {2} & 18000.00 &\bf 9490.74 & 100.00\% & \bf20.01\% & 10892069 & \bf363   & 2483  &\bf 83 \\
      &       & {3} & 18000.00 & 18000.00 & 100.00\% & \bf68.93\% & 10139295 & \bf214   & 1537  &\bf 75 \\\hline
\multicolumn{3}{|c}{\bf Average 50: } &\bf 16145.28 & 16946.05 & 69.96\% &\bf 51.35\% & 8471341 &\bf 255   & 2791  & \bf652 \\\hline
\multicolumn{3}{|c}{\bf Total Average: } &\bf11231.69&	11409.54	&51.49\%	&\bf29.93\%	&7735135 &	\bf3287&	1718	&\bf773\\\hline
\end{tabular}}
     \caption{Results for synthetic instances for vertical distance. \label{tab:SynthV}}%
\end{table}

 \begin{table}
  \centering{\small\begin{tabular}{|c|c|c||rr|rr|rr|rr|}\cline{4-11}
    \multicolumn{3}{c||}{} & \multicolumn{2}{c|}{\texttt{Time} (secs.)} &  \multicolumn{2}{c|}{\texttt{GAP}} &  \multicolumn{2}{c|}{\texttt{Nodes}}  &  \multicolumn{2}{c|}{\texttt{Memory} (MB)} \\\hline
   $n$ &$p$ & $d$ & Compact & B\&P &  Compact & B\&P & Compact & B\&P &Compact & B\&P\\\hline
\multirow{5}{*}{20} & \multirow{3}{*}{2} & {2} & 2799.81 & \bf1413.43 & \bf0.03\% & 1.50\% & 4193628 & \bf545   & \bf216   & 424 \\
      &       & {3} & 10649.52 & \bf4226.37 & 8.41\% &\bf 0.01\% & 17136729 & \bf617   & 696   & \bf693 \\
      &       & {8} & 18000.00 &\bf 14942.66 & 100.00\% & \bf36.68\% & 32145208 & \bf324   & 488   & \bf476 \\\cline{2-11}
      & \multirow{2}{*}{5} & {2} & 17977.28 & \bf3686.10 & 100.00\% &\bf 4.15\% & 35717983 & \bf5312  & 9878  &\bf 922 \\
      &       & {3} & 18000.00 &\bf 4610.48 & 100.00\% & \bf16.99\% & 39208660 &\bf 3325  & 5832  &\bf 52 \\\hline
\multicolumn{3}{|c}{\bf Average 20:} & 13485.35 & \bf5775.81 & 61.69\% & \bf11.87\% & 25680442 & \bf2025  & 3422  & \bf514 \\\hline
\multirow{6}{*}{30} & \multirow{3}{*}{2} & {2} & 11021.12 & \bf10263.53 & 43.06\% &\bf 4.35\% & 11908005 &\bf 293   & 2283  &\bf 1230 \\
      &       & {3} & \bf13503.52 & 15061.10 & 50.87\% & \bf18.57\% & 13408414 & \bf178   & 2644  &\bf 894 \\
      &       & {8} & 18000.00 & 18000.00 & 100.00\% & \bf85.77\% & 18682135 & \bf14    & 3855  & \bf171 \\\cline{2-11}
      & \multirow{2}{*}{5} & {2} & 18000.00 & \bf9923.88 & 100.00\% & \bf14.72\% & 24028581 &\bf 1361  & 11038 & \bf1765 \\
      &       & {3} & 18000.00 &\bf 12882.81 & 100.00\% &\bf 22.45\% & 23654639 & \bf631   & 9211  & \bf785 \\\cline{2-11}
      & \multicolumn{1}{c|}{10} & {2} & 17998.01 &\bf 4745.50 & 100.00\% & \bf14.56\% & 24603724 &\bf 1162  & 8567  & \bf169 \\\hline
\multicolumn{3}{|c}{\bf Average 30:} & 16087.22 & \bf11815.89 & 82.32\% &\bf 26.74\% & 19380916 &\bf 607   & 6266  &\bf 835 \\\hline
\multirow{7}{*}{40} & \multirow{3}{*}{2} & {2} & \bf13500.52 & 17798.69 & 63.24\% & \bf14.27\% & 8495764 &\bf 76    & 2259  & \bf1074 \\
      &       & {3} &\bf 13509.64 & 18000.00 & 64.33\% & \bf31.75\% & 7710536 & \bf25    & 2929  &\bf 601 \\
      &       & {8} & 18000.00 & 18000.00 & 100.00\% & \bf73.86\% & 12839616 & \bf2     & 4272  & \bf192 \\\cline{2-11}
      & \multirow{2}{*}{5} & {2} & 18000.00 & 18000.00 & 100.00\% & \bf42.62\% & 17696588 & \bf193   & 10524 & \bf904 \\
      &       & {3} & 18000.00 & 18000.00 & 100.00\% & \bf69.83\% & 19758253 &\bf 76    & 9542  & \bf378 \\\cline{2-11}
      & \multirow{2}{*}{10} & {2} & 17688.37 & \bf10787.01 & 100.00\% & \bf22.17\% & 17749077 & \bf212   & 8306  & \bf45 \\
      &       & {3} & 18000.00 & 18000.00 & 100.00\% & \bf63.85\% & 16865456 & \bf111   & 5522  &\bf 44 \\\hline
\multicolumn{3}{|c}{\bf Average 40:} & \bf16671.31 & 16943.82 & 89.65\% & \bf45.48\% & 14445041 &\bf 99    & 6193  & \bf463 \\\hline
\multirow{8}{*}{45} & \multirow{3}{*}{2} & {2} & \bf13500.65 & 17927.36 & 61.73\% & \bf16.31\% & 7295436 & \bf26    & 2189  &\bf 1101 \\
      &       & {3} & \bf13507.23 & 18000.00 & 74.40\% & \bf25.74\% & 5456543 &\bf 10    & 2359  & \bf602 \\
      &       & {8} & 18000.00 & 18000.00 & 100.00\% & \bf69.62\% & 10392241 & \bf1     & 4242  &\bf 243 \\\cline{2-11}
      & \multirow{3}{*}{5} & {2} & 18000.00 &\bf 16684.12 & 100.00\% & \bf45.50\% & 13916126 &\bf 714   & 9640  &\bf 729 \\
      &       & {3} & 18000.00 & 18000.00 & 100.00\% &\bf 68.74\% & 15523670 & \bf31    & 9216  & \bf294 \\
      &       & {8} & 18000.00 & 18000.00 & 100.00\% & 100.00\% & 13168789 &\bf 1     & 1284  &\bf 57 \\\cline{2-11}
      & \multirow{2}{*}{10} & {2} & 18000.00 & \bf13383.58 & 100.00\% & \bf24.11\% & 14582103 &\bf 234   & 8573  & \bf61 \\
      &       & {3} & 18000.00 & \bf16757.30 & 100.00\% & \bf60.65\% & 14067756 & \bf54    & 6130  &\bf 59 \\\hline
\multicolumn{3}{|c}{\bf Average 45:} & \bf16876.16 & 17096.84 & 92.02\% & \bf51.33\% & 11800333 & \bf134   & 5454  &\bf 393 \\\hline
\multirow{8}{*}{50} & \multirow{3}{*}{2} & {2} & \bf13500.40 & 15300.83 & 71.34\% &\bf 10.99\% & 6941692 &\bf 11    & 3250  &\bf 950 \\
      &       & {3} & \bf13512.06 & 18000.00 & 60.70\% &\bf 13.84\% & 5691897 &\bf 5     & 1877  &\bf 1068 \\
      &       & {8} & 18000.00  & 18000.00 & 100.00\% &\bf 64.78\% & 9864048 &\bf 1     & 3915  &\bf 345 \\\cline{2-11}
      & \multirow{3}{*}{5} & {2} & 18000.00 & \bf15543.12 & 100.00\% &\bf 46.58\% & 15611084 & \bf492   & 10326 & \bf704 \\
      &       & {3} & 18000.00 & 18000.00 & 100.00\% & \bf85.37\% & 14771026 &\bf 9     & 8255  & \bf199 \\
      &       & {8} & 18000.00 & 18000.00 & 100.00\% & 100.00\% & 10436698 &\bf 1     & 3278  &\bf 69 \\\cline{2-11}
      & \multirow{2}{*}{10} & {2} & 18000.00 & \bf15309.03 & 100.00\% & \bf29.85\% & 14012908 &\bf 273   & 6934  &\bf 72 \\
      &       & {3} & 18000.00 & 18000.00 & 100.00\% &\bf 67.47\% & 12123118 &\bf 26    & 5274  &\bf 69 \\\hline
\multicolumn{3}{|c}{\bf Average 50:} & \bf16876.66 & 17023.10 & 91.50\% & \bf52.36\% & 11181559 & \bf102   & 5388  & \bf434 \\\hline
\multicolumn{3}{|c}{\bf Total Average:} &   16038.45&	\bf13854.81&83.78\%	&\bf37.41\%	&16590341	&\bf569	 &5435	&\bf531\\\hline
\end{tabular}}
     \caption{Results for synthetic instances for $\ell_1$ distance. \label{tab:SynthL1}}%
\end{table}%

\section{Scalability: Bounding the error in aggregation procedures\label{sec:6}}

This section is devoted to analyze the issue of scalability of our approach. We are aware that the
methodology based on a branch and price algorithm may be computationally costly (we refer the reader
to the Section \ref{sec:comp_result} for further details).  For that reason, we derive an approach that allows one to handle large data sets with appropriate error bounds.

Our approach is based on aggregating data to reduce the dimensionality of the original problem so
that our branch and price approach can properly handle the problem. The important issue is that we
can provide error bounds on these approximations  that monotonically decrease with the quality of the aggregation. Obviously, aggregation strategies are not new since they have been already applied in some other areas although mostly from a heuristic point of view (see e.g., \cite{CurrentSchilling87,CurrentSchilling90})

Let $X=\{x_1,\ldots,x_n\}\subset \mathbb{R}^d$ be a set of demands points. Aggregating $X$ into a new set of demand points $X^\prime$  consists of replacing  $X$ by $X^\prime=\{x_1^\prime, \ldots, x_n^\prime\}$ and to assign each point $x_i$ in $X$ to a point $x^\prime_i$ in $X^\prime$ (since usually the cardinality of the different elements of $X^\prime$ is smaller than the cardinality of $X$, several $x_i$ may be assigned to the same $x_i^\prime$ and thus actually, some of the elements in $X^\prime$ coincide). A possible choice can be substituting the set of original demand points by the centroids obtained by any of the available clustering techniques. In any case, when solving \eqref{mofhp0} for $X^\prime$ instead of using $X$ one incurs in aggregation errors.

Let $\mathbb{H}$ be the optimal arrangement of $p$ hyperplanes for the problem and $\mathbf{e} = (e_1, \ldots, e_n)$ with $e_i = \varepsilon_{x_i}\Big(\mathbb{H}\Big)$, for $i\in I$, the residuals with respect to $\mathbb{H}$. Analogously, let $\mathbb{H}^\prime$ be the optimal arrangement for the demand points in $X^\prime$ and $\mathbf{e}^\prime$ the vector of residuals.

\begin{theo}
Let $T = \dmax_{i=1,\ldots,n}  \d(x_i,x_i^\prime)$. Then, the following relation holds:
\begin{equation}
| \omf ( \mathbf{e}^\prime)- \omf (\mathbf{e})| \le 2  \omf (T, \ldots, T).
\end{equation}
\end{theo}
\begin{proof}
  First of all, observe that, based on the triangular inequality, for any $\mathbb{H}$
$$
\varepsilon_{x_i}(\mathbb{H})  \le \varepsilon_{x_i^\prime}(\mathbb{H})+ \d(x_i,x_i^\prime), \forall i =1, \ldots, n.
$$
Let us also consider the vector $\mathbf{t}=(\d(x_1,x_1^\prime),\ldots, \d(x_n,x_n^\prime))$ of distances from the original points in $X$ to their corresponding points in $X^\prime$ and denote by $\tilde{\mathbf{e}}=(\varepsilon_{x_1^\prime}(\mathbb{H}), \ldots, \varepsilon_{x_n^\prime}(\mathbb{H}))$. Since the function \textsf{OM} is non-decreasing monotone and sublinear, it follows that:
$$
\textsf{OM}_\lambda (\mathbf{e}) \le \textsf{OM}_\lambda (\tilde{\mathbf{e}}+\mathbf{t})   \le \textsf{OM}_\lambda (\tilde{\mathbf{e}}) + \textsf{OM}_\lambda (\mathbf{t}).
$$
Hence, since $T \geq \d(x_i,x_i^\prime)$ for all $i =1, \ldots, n$, we get that:
\begin{equation*}
|\textsf{OM}_\lambda (\mathbf{e}) - \textsf{OM}_\lambda (\tilde{\mathbf{e}})|\le \textsf{OM}_\lambda(T,\ldots,T).
\end{equation*}
From the above inequality we can apply \cite[Theorem 5]{Geo77}  to conclude that
\begin{center}
$|  \textsf{OM}_\lambda ( \mathbf{e}^\prime)- \textsf{OM}_\lambda (\mathbf{e})| \le 2  \textsf{OM}_\lambda (T, \ldots, T).
$
\end{center}
\end{proof}

The difference considered in the above theorem is the excess due to the implementation of an approximate solution based on the reduced model with data set $X^\prime$ rather than the correct optimal solution for the original data in the larger set $X$. This result allows us to scale our CG algorithm to problems of any size using aggregation techniques and providing estimates on the deviation from the optimal value.

We illustrate the application of the above result including the percent error obtained aggregating to 20 points some of our random problems with 50 points by the $20$-mean clustering technique. As one can see in Table \ref{t:scalan2} the  percent errors are small. Observe that in some cases they are even negative, for problems that were not solved to optimality, and where the hyperplanes obtained by  aggregating points, once evaluated on the actual 50 points, produce a smaller error than the upper bound found by the algorithm on the original dataset.

\renewcommand{\arraystretch}{1.2}
\setlength{\tabcolsep}{5pt}
\begin{table}
\centering
\begin{tabular}{|c|c|c|c|}\cline{3-4}
\multicolumn{2}{c}{} & 	\multicolumn{2}{|c|}{\texttt{error}   (\%)}		\\\hline
	$p$ 				& \texttt{type} 	 &$d=2$ & $d=3$\\\hline
	\multirow{4}{*}{2} 	&	W		&	3.48	& 2.78	\\
					&	K		&-17.84	&	-16.52	\\
					&	C		&	4.40	& 	5.90	\\
					&	D		&	3.59	&	2.87	\\\hline
	\multirow{4}{*}{5} 	&	W		&	-0.16	&	1.21	\\
					&	K		&	-9.17	&	-2.99	\\
					&	C		&	6.60	&	20.86	\\
					&	D		&	0.16	& 	-11.16	\\\hline
\end{tabular}
\caption{\% aggregation errors for 50 points problems and vertical distance.\label{t:scalan2}}
\end{table}

\section{Conclusions\label{sec:7}}

This paper considers the problem  of locating a given number of hyperplanes in order to minimize an objective function of the distances from a set of points. Each point is assigned to its closest hyperplane, thus inducing  as many clusters  as the number of fitting hyperplanes. The distance from each point to its corresponding fitting hyperplane can be  seen as a residual and these residuals are aggregated using ordered median functions that are ordered weighted averages representing different types of utilities. Two exact approaches are presented to solve the problem. The first one is based on a compact mixed integer formulation whereas the second one is an extended set partitioning formulation with an exponential number of variables that is handled by a branch-and-price approach. To enhance the performance of this last method we have developed a generator of initial feasible solutions  based on  geometrical properties of the optimal solutions of the hyperplane location problem that we have also derived in this paper, and that are used to initialize the column generation routine of this branch-and-price. We have also presented  a heuristic pricing strategy that is used in combination with the exact one to speed up some pricing iterations. We report the comparison of both method to solve the problem in two different datasets on an extensive battery of computational experiments. The issue of scalability of the exact methods is also analyzed obtaining theoretical upper bounds of the error induced by some aggregated versions of the original dataset.

A possible extension to be developed in a follow up paper is the development of alternative heuristic algorithms capable to solve the problem for large instances. In view of the applications of the proposed methodology in machine learning, other types of tools could be also explored under the multisource ordered median paradigm, as for instance Support Vector Machines, where a first attempt have been already proposed by \cite{BJP19}.

\section*{Acknowledgements}

This research has been partially supported by Spanish Ministry of Econom{\'\i}a and  Competitividad/FEDER grants number MTM2016-74983-C02-01.

\end{document}